\documentclass[12pt]{amsart}
\usepackage{amssymb}
\usepackage[mathscr]{eucal}
\usepackage{epsf}
\usepackage{epsfig}
\usepackage{color}
\DeclareGraphicsExtensions{.pstex,.eps,.epsi,.ps}

 \newtheorem{thm}{Theorem}[section]
 \newtheorem{cor}[thm]{Corollary}
 \newtheorem{lem}[thm]{Lemma}
 \newtheorem{prop}[thm]{Proposition}
 \theoremstyle{definition}
 
 \theoremstyle{remark}
 \newtheorem{rem}[thm]{Remark}
 \newtheorem{prob}[thm]{Problem}

\newcommand{\Real}{\mathbb{R}}

\newcommand{\ver}{\mathcal{V}}
\newcommand{\hor}{\mathcal{H}}
\newcommand{\ric}{\text{Rc}}
\newcommand{\vol}{\text{vol}}

\newcommand{\tr}{\textbf{tr}}
\newcommand{\hess}{\nabla^2}

\newcommand{\E}{\mathcal E}
\newcommand{\ham}{\mathbf H}

\newcommand{\x}{\mathbf x}
\newcommand{\R}{\mathbf R}
\newcommand{\p}{\mathbf p}
\newcommand{\diver}{\mathbf{div}}
\newcommand{\lag}{\mathbf L}
\newcommand{\F}{\mathbf F}
\newcommand{\cF}{\mathcal F}
\newcommand{\m}{\mathfrak m}
\newcommand{\f}{\mathbf f}

\newcommand{\cP}{\mathcal P}
\newcommand{\K}{\mathbf K}
\newcommand{\leg}{\mathbb L}
\newcommand{\ve}{\mathbf v}
\newcommand{\w}{\mathbf w}

\newcommand{\marginnote}[1]
{
}

\begin{document}

\title[Displacement interpolations]{Displacement interpolations from a Hamiltonian point of view}
\author{Paul W.Y. Lee}
\email{wylee@math.cuhk.edu.hk}
\address{Room 216, Lady Shaw Building, The Chinese University of Hong Kong, Shatin, Hong Kong}

\date{\today}

\maketitle

\begin{abstract}
One of the most well-known results in the theory of optimal transportation is the equivalence between the convexity of the entropy functional with respect to the Riemannian Wasserstein metric and the Ricci curvature lower bound of the underlying Riemannian manifold. There are also generalizations of this result to the Finsler manifolds and manifolds with a Ricci flow  background. In this paper, we study displacement interpolations from the point of view of Hamiltonian systems and give a unifying approach to the above mentioned results.
\end{abstract}


\section{Introduction}

Due to its connections with numerous areas in mathematics, the theory of optimal transportation has gained much popularity in recent years. In this paper, we will focus on the, so called, displacement convexity in the theory of optimal transportation and its connections with Ricci curvature lower bounds and the Ricci flow.

Let $M$ be a manifold. Recall that the optimal transportation problem corresponding to the cost function $c:M\times M\to\Real$ is the following minimization problem:
\begin{prob}\label{prob:OT}
Let $\mu$ and $\nu$ be two Borel probability measures on $M$. Minimize the following total cost
\[
\int_Mc(x,\varphi(x))d\mu(x),
\]
among all Borel maps $\varphi$ which pushes $\mu$ forward to $\nu$: $\varphi_*\mu=\nu$. Minimizers of this problem are called optimal maps.
\end{prob}

In this paper, we are interested in cost functions defined by minimizing action functionals. More precisely, let $\lag:\Real\times TM\to\Real$ be a smooth function, called Lagrangian. The function $\lag$ defines a family of cost functions $c_T:M\times M\to\Real$ given by
\begin{equation}\label{prob:cost}
c_T(x,y)=\inf\limits_{\gamma(0)=x,\gamma(T)=y}\int_0^T \lag(s,\gamma(s),\dot\gamma(s))ds,
\end{equation}
where the infimum is taken over all $C^1$ curves $\gamma:[0,T]\to M$ joining
$x$ and $y$: $\gamma(0)=x$ and $\gamma(T)=y$.

Assume that the Lagrangian $\lag$ is fibrewise strictly convex, super-linear, and the corresponding Hamiltonian flow is complete. Then the infimum in (\ref{prob:cost}) is achieved. Moreover, if we also assume that the measure $\mu$ is absolutely continuous with respect to the Lebesgue class, then the optimal transportation problem corresponding to the cost function $c_T$ has a minimizer $\varphi_T$ which is unique $\mu$-almost-everywhere \cite{Br1,Mc2,BeBu,FaFi,Fi}.

As the time $T$ varies, the optimal transportation problem defines a one parameter family of optimal maps $\varphi_T$. This, in turn, defines a one parameter family of probability measures $\mu_t:=(\varphi_t)_*\mu$. These curves in the space of Borel probability measures, first introduced in \cite{Mc1,Mc2}, are called
McCann's displacement interpolations. In \cite{Br2}, displacement interpolations are called generalized geodesics and they were used to obtain a far reaching extension of the optimal transportation problem.

The connection between the optimal transportation problem and the Ricci curvature lies in the convexity of the classical entropy functional along displacement interpolations. More precisely, the entropy functional $\mathcal E_1:\mathcal P_{ac}\to\Real$ is defined on the space $\mathcal P_{ac}$ of Borel probability measures which are absolutely continuous with respect to the Lebesgue class by
\begin{equation}\label{ent1}
\mathcal E_1(\mu)=\int_M\rho(x)\log\rho(x) d\vol(x)
\end{equation}
where $\mu=\rho\,\vol$ and $\vol$ is the Riemannian volume form of $g$.

Let $g$ be a Riemannian metric on the manifold $M$ and let $\lag$ be the Lagrangian defined, in local coordinates, by
\begin{equation}\label{kinetic}
\lag(x,v)=\frac{1}{2}\sum_{i,j=1}^ng_{ij}v_iv_j
\end{equation}
where $g_{ij}=g(\partial_{x_i},\partial_{x_j})$.

It was shown in \cite{OtVi,CoMcSc1,StVo} that the functional $\E_1$ is convex along any displacement interpolation defined by the cost function (\ref{prob:cost}) with Lagrangian (\ref{kinetic}) if and only if the Ricci curvature defined by the Riemannian metric $g$ is non-negative.

Instead of the entropy functional $\E_1$, one can also consider the relative entropy $\E_2$ defined by
\begin{equation}\label{ent2}
\mathcal E_2(\mu)=\int_M\rho(x)\log\rho(x) e^{-\mathfrak U(x)}d\mu(x),
\end{equation}
where $\mathfrak U:M\to\Real$ is a smooth function. It was shown in \cite{OtVi,CoMcSc2,St} that the entropy $\E_2$ is convex along any displacement interpolation defined by the cost function (\ref{prob:cost}) with Lagrangian (\ref{kinetic}) if and only if the Bakry-Emery tensor $\ric+\hess \mathfrak U$ is non-negative.

The above results were also generalized to the Finsler case by \cite{Oh}. In this case, (\ref{kinetic}) is replaced by Lagrangians $\lag=\lag(t,x,v)$ which are homogeneous of degree two in the $v$-variable. There are also various generalizations of the above results to manifolds with a Ricci flow background \cite{McTo,To,Lo,Bre}.

The Euler-Lagrange equation of the action functional (\ref{prob:cost}) is given by a Hamiltonian system on the cotangent bundle $T^*M$ of the manifold $M$. In this paper, we study the displacement interpolations from the point of view of this Hamiltonian structure. By using the curvature of Hamiltonian systems introduced in \cite{AgGa}, we prove a differential inequality (Theorem \ref{main}) which recover, unify, and generalize the above mentioned results on the Ricci curvature and the Ricci flow.

The structure of this paper is as follows. In section \ref{MR}, the backgrounds and the statements of the main results in this paper are introduced. In section \ref{COHS}, we recall the definitions and properties of curvatures for Hamiltonian systems. In section \ref{OPAHJE}, we recall several notions in the optimal transportation theory and its connections with Hamiltonian systems and the Hamilton-Jacobi equation. In section \ref{NMH} and \ref{NMHWERM}, we compute the curvature of some Hamiltonian systems studied in this paper. The rest of the sections are devoted to the proofs of the results.

\section*{Acknowledgment}

After a major revision of this paper, the author received the following message from Professor Juan-Carlos Alvarez Paiva ``Agrachev's construction of curvature for convex Hamiltonian systems is completely equivalent (although not obviously so) to the constructions of Griffone \cite{Gr} and Foulon \cite{Fo}, while the analysis of the geometry of curves of Lagrangian subspaces and its relation to curvature is due to S. Ahdout \cite{Ah}.'' I would like to thank him for his comments.

\smallskip

\section{Main Results}\label{MR}

In this section, we give the background and statements of the main results. For this, let us first discuss the minimization problem in (\ref{prob:cost}). For simplicity, we assume that the manifold $M$ is compact. In order to ensure the existence of minimizers in (\ref{prob:cost}), we make the following assumptions on the Lagrangian $\lag$ throughout this paper without mentioning.
\begin{itemize}
\item $\lag$ is fibrewise uniformly convex (i.e. $v\mapsto \lag(t,x,v)$ has positive definite Hessian for each time $t$ in the interval $[0,T]$ and each point $x$ in $M$),
\item $\lag$ is super-linear (i.e. there is a Riemannian metric $|\cdot|$ and positive constants $C_1$ and $C_2$ such that $\lag(t,x,v)\geq C_1|v|-C_2$ for all time $t$ in $[0,T]$ and all tangent vector $(x,v)$ in $TM$).
\end{itemize}
Under the above assumptions, the minimizers of (\ref{prob:cost}) exist. Moreover, under additional assumption, they can be described as follows.

Let $\ham:\Real\times T^*M\to\Real$ be the Hamiltonian defined by the Legendre transform of $\lag$
\[
\ham(t,x,p)=\sup_{v\in T_xM}\left[p(v)-\lag(t,x,v)\right].
\]

Let $\vec\ham$ be the Hamiltonian vector field of $\ham$ defined on the cotangent bundle $T^*M$ by
\[
\vec\ham =\sum_{i=1}^n\left(\ham_{p_i}\partial_{x_i}-\ham_{x_i}\partial_{p_i}\right).
\]
We also make the following assumption throughout this paper without mentioning.
\begin{itemize}
\item the time dependent flow $\Phi_{t,s}$ of $\vec\ham$ defined by $\frac{d}{dt}\Phi_{t,s}=\vec\ham(\Phi_{t,s})$ and $\Phi_{s,s}(x,p)=(x,p)$ is complete.
\end{itemize}

We also set $\Phi_t=\Phi_{t,0}$. Under the above assumptions, the minimizers of (\ref{prob:cost}) are given by the projections $t\mapsto\pi(\Phi_t(x,p))$ of trajectories $t\mapsto\Phi_t(x,p)$ to the base $M$. Here $\pi:T^*M\to M$ is the natural projection. Moreover, if we assume that the measure $\mu$ is absolutely continuous with respect to the Lebesgue class, then Problem \ref{prob:OT} with cost given by (\ref{prob:cost}) also has a solution. More precisely, we have the following well-known result.

\begin{thm}\label{OTexist}\cite{Br1,Mc2,BeBu}
Assume that $M$ is a closed manifold and the measure $\mu$ is absolutely continuous with respect to the Lebesgue class. Then there is a solution $\varphi_T$, which is unique $\mu$-almost-everywhere, to Problem \ref{prob:OT} with cost $c_T$ given by (\ref{prob:cost}). Moreover, there is a locally semi-concave function $\f$ such that
\[
\varphi_T(x)=\pi(\Phi_T(d\f_x)).
\]
\end{thm}

The function $\f$ in Theorem \ref{OTexist} is the potential of the optimal map $\varphi_T$. For each such potential $\f$, one can define $\varphi_t$, $0<t<T$, by
\[
\varphi_t(x)=\pi(\Phi_t(d\f_x)).
\]
Then $\varphi_t$ is the optimal map between the measures $\mu$ and
\begin{equation}\label{disint}
\mu_t:=(\varphi_t)_*\mu
\end{equation}
for the cost $c_t$.
The family of measures $\mu_t$ defined in (\ref{disint}) is called a displacement interpolation corresponding to the cost (\ref{prob:cost}). It was first introduced by \cite{Mc1,Mc2} in the Euclidean setting and it turns out to be the key to various connections of the optimal transportation problem with differential geometry. For instance, in the Riemannian case, the Ricci curvature of a manifold being non-negative is equivalent to the convexity of the entropy functional (\ref{ent1}) along displacement interpolations corresponding to the cost function given by the square of the Riemannian distance.

Let $\cP_{ac}$ be the space of all Borel probability measures which are absolutely continuous with respect to the Lebesgue class. In this paper, we mainly focus on functionals of the following form. Let $\m_t$ be a family of volume forms on $M$ which vary smoothly in $t$. We will denote the measure induced by $\m_t$ using the same symbol. Let us fix a function $\cF:\Real\to\Real$. The measures $\m_t$ and the function $\cF$ define a functional $\F$ on the space $[0,T]\times \cP_{ac}$ by
\begin{equation}\label{fnal}
\F(t,\mu)=\int_M\cF(\rho(t,x))d\mu(x),
\end{equation}
where $\mu=\rho(t,\cdot)\m_t$.

We show that the monotonicity and convexity properties of the functional defined in (\ref{fnal}) are related to the curvature $\R$ and the volume distortion $v$ of the Hamiltonian system $\vec\ham$. Here we give a brief introduction to these two concepts. A detail discussion can be found in Section \ref{COHS} and \ref{POT}.

Recall that the cotangent bundle $T^*M$ of the manifold $M$ is equipped with a symplectic form $\omega=\sum_{i=1}^ndp_i\wedge dx_i$ and the flow of the time-dependent Hamiltonian vector field $\vec\ham$ is denoted by $\Phi_{t,s}$. Let $\pi:T^*M\to M$ be the natural projection and let $\ver$ be the vertical bundle defined as the kernel of the map $d\pi$. For each point $(x,p)$ in the cotangent bundle $T^*M$ of the manifold $M$, one can define a moving frame $E^{t,s}:=(e_1^{t,s},...,e_n^{t,s})^T$ and $F^{t,s}:=(f_1^{t,s},...,f_n^{t,s})^T$, called a canonical frame, of the symplectic vector space $T_{(x,p)}T^*M$ such that the following conditions are satisfied for each $i$ and $j$:
\begin{itemize}
\item $e_i^{t,t}$ is contained in the vertical space $\ver_{(x,p)}$ for each time $t$,
\item $\dot e_i^{t,s}=f_i^{t,s}$,
\item $\ddot e_i^{t,s}$ is also contained in $\ver_{(x,p)}$,
\item the canonical frame $E^{t,s}, F^{t,s}$ is a Darboux basis for each $(t,s)$ (i.e. $\omega(f_i^{t,s},e_j^{t,s})=\delta_{ij}$, $\omega(f_i^{t,s},f_j^{t,s})=\omega(e_i^{t,s},e_j^{t,s})=0$).
\end{itemize}
Here and throughout this paper, we use the following notational convention: dot always denotes the derivative with respect to time $t$, $s$, or $\tau$. For instance $\dot V^t=\frac{d}{dt} V^t$. In the case of the canonical frame, both $E^{t,s}$ and $F^{t,s}$ contain two time parameters $t$ and $s$. In this case, dot denotes the derivative with respect to the first parameter. For instance $\dot e_i^{t,s}=\frac{d}{dt} e_i^{t,s}$ and $\ddot e_i^{t,s}=\frac{d^2}{dt^2}e_i^{t,s}$.

The curvature operator $\R^t_{(x,p)}:\ver_{(x,p)}\to\ver_{(x,p)}$ of the Hamiltonian system $\vec\ham$ is the linear operator defined by
\[
\R^t_{(x,p)}(e_i^{t,t})=-\ddot e_i^{t,t}=-\dot f_i^{t,t}.
\]
Finally, the volume distortion $v:\Real\times T^*M\to\Real$ compares the volume forms $\m_t$ with the frame $F^{t,t}$. More precisely,
\[
v(t,x,p):=\log\left(\pi^*\m_t(f_1^{t,t},...,f_n^{t,t})\right).
\]

Curvature of Hamiltonian systems was first introduced and studied in \cite{AgGa}. It serves as a Hamiltonian analogue of the curvature operator in Riemannian geometry. More precisely, let us fix a Riemannian metric $g$ and let $\ham$ be the kinetic energy given in a local coordinate chart by
\begin{equation}\label{kin}
\ham(x,p)=\sum_{i,j=1}^n\frac{1}{2}g^{ij}(x)p_ip_j
\end{equation}
where $g^{ij}$ is the inverse matrix of $g_{ij}=g(\partial_{x_i},\partial_{x_j})$. In this case, the curvature operator $\R^t$ is independent of time $t$ and it is, under some identifications using the Riemannian metric $g$, the curvature operator. Suppose that $\m_t=e^{-\mathfrak U}\vol$, where  $\mathfrak U:M\to\Real$ is a smooth function and $\vol$ is the Riemannian volume form of $g$. Then the volume distortion is simply $v=-\pi^*\mathfrak U$.

Next, we consider a particular case of the functional $\F$ with $\cF(r)=\log(r)$. In this case, the functional $\F$ is a time-dependent version of the relative entropy (\ref{ent2}) and we obtain the following theorem which unify several results in the literature.

\begin{thm}\label{main}
Assume that $\cF(r)=\log(r)$ and let $b$ be any function of time $t$. Then the functional $\F$ satisfies the following inequality:
\begin{equation}\label{conclude}
\begin{split}
&\frac{d^2}{dt^2}\F(t,\mu_t)+b(t)\frac{d}{dt}\F(t,\mu_t)+\frac{nb(t)^2}{4}\geq \int_M\Big(\tr(\R^t_{\Phi_t(\x,d\f)})\\ &-\frac{d^2}{dt^2}v(t,\Phi_t(\x,d\f)) -b(t)\frac{d}{dt}v(t,\Phi_t(\x,d\f))\Big) d\mu(\x)
\end{split}
\end{equation}
for each smooth displacement interpolation $\mu_t$ corresponding to the cost (\ref{prob:cost}).
\end{thm}

In Theorem \ref{main}, we assume that $\mu_t$ is a smooth displacement interpolation (see Section \ref{OPAHJE} for the definition). This allows us to focus on the ideas of the proof and avoids technical difficulties that arise due to the lack of regularity of the potential $\f$. These technicalities can be dealt with along the lines in \cite{CoMcSc1,To} and most of them follows from the general results in \cite{Vi2}. The details will be reported elsewhere.

By specializing Theorem \ref{main} to the Riemannian case, we recover several well-known results. More precisely, let $\ham$ be the kinetic energy defined by (\ref{kin}). Let $b\equiv 0$ and $\m_t=e^{-\mathfrak U}\vol$. Then the functional $\F$ defined in (\ref{fnal}) coincides with the relative entropy $\E_2$ defined in (\ref{ent2}) and we recover the following result which appeared in \cite{OtVi,CoMcSc2,StVo,St} (see section \ref{POT} for some partial results for more general Hamiltonian systems and also section \ref{finsler} for a discussion on the Finsler case).

\begin{cor}\label{cor1}
Let $g$ be a Riemannian metric with Ricci curvature $\ric$. Then the Bakry-Emery tensor $\ric+\hess \mathfrak U$ satisfies the following condition
\begin{equation}\label{BE}
\ric+\hess \mathfrak U\geq Kg
\end{equation}
if and only if
\begin{equation}\label{CD}
\frac{d^2}{dt^2}\E_2(\mu_t)\geq \frac{2K}{T} C_T(\mu,\nu)
\end{equation}
for each smooth displacement interpolation $\mu_t$ corresponding to the cost (\ref{prob:cost}) with Lagrangian defined by (\ref{kinetic}). Here $\hess\mathfrak U$ denotes the Hessian of the function $\mathfrak U$ with respect to the given Riemannian metric $g$. $C_T(\mu,\nu)$ is defined by the optimal transportation problem
\[
C_T(\mu,\nu):=\inf_{\varphi_*\mu=\nu}\int_Mc_T(x,\varphi(x))d\mu(x),
\]
where the infimum is taken over all Borel maps $\varphi$ which pushes $\mu$ forward to $\nu$.
\end{cor}

We remark that even though we replace the usual displacement interpolations by smooth displacement interpolations in Corollary \ref{cor1}, the equivalence between the Bakry-Emery condition (\ref{BE}) and the convexity of the functional $\F$ still holds.

We also remark that the condition (\ref{CD}) does not involve any differential structure and it was used in \cite{St,LoVi} to define Ricci curvature lower bound for more general metric measure spaces.

If we replace the Hamiltonian $\ham$ in (\ref{kin}) by a Finsler Hamiltonian (i.e. any Hamiltonian which is homogeneous of degree two in the $p$ variable), set $b\equiv 0$ in Theorem \ref{main}, and assume that $\mathfrak m_t$ is independent of time $t$, then the conclusion of Theorem \ref{main} becomes
\begin{equation}\label{finslerest}
\begin{split}
&\frac{d^2}{dt^2}\F(t,\mu_t)\geq \int_M\Big(\tr(\R_{\Phi_t(\x,d\f)})-\mathcal L^2_{\vec \ham}v(t,\Phi_t(\x,d\f))\Big) d\mu(\x),
\end{split}
\end{equation}
where $\mathcal L_V$ denotes the Lie derivative in the direction $V$. In section \ref{finsler}, we will check that $\R$ coincides (up to some identifications) with the Riemann curvature in Finsler geometry. It follows that the $\infty$-Ricci curvature lower bound condition in \cite{Oh} is the same as
\[
\tr(\R)-\mathcal L^2_{\vec \ham}v\geq K\ham.
\]
Therefore, Theorem \ref{main} recovers the corresponding results in \cite{Oh} (see also Theorem \ref{sidethm}). However, we remark that $\ham$ is only $C^1$ on the zero section of $T^*M$ in the Finsler case. Because of this, additional technical difficulties arise for results concerning usual displacement interpolations (see \cite{Oh}).

By considering functionals more general than (\ref{fnal}), one can recover various other known results which are not covered by Theorem \ref{main}. We will not pursue this here. Instead, we consider Hamiltonians of the following form which motivated the whole work
\begin{equation}\label{mech}
\ham(x,p)=\frac{1}{2}\sum_{i,j}g^{ij}(t,x)p_ip_j+U(t,x).
\end{equation}
Here both the metric $g$ and the potential $U$ depend on time $t$.

We would like to apply Theorem \ref{main} to this Hamiltonian and recover some results between the optimal transportation theory and the Ricci flow (c.f. \cite{McTo,To,Lo,Bre}). However, in all these works, a Hamiltonian $\ham$ is fixed and the functional $\F$ is shown to have certain convexity properties along displacement interpolations corresponding to the cost function (\ref{prob:cost}) with Hamiltonian $\ham$. Instead we propose the following approach of using Theorem \ref{main} to find Hamiltonians $\ham$ such that the functional $\F$ is monotone or convex along displacement interpolations corresponding to $\ham$.

First, we create some functional parameters to get rid of the undesire terms in (\ref{conclude}). On the other hand, $g$ should be related to the Ricci flow. Therefore, it is natural to pick the time dependent metric $g$ satisfying the following equation:
\begin{equation}\label{timeRiem}
\dot g=c_1\ric+c_2g,
\end{equation}
where $\ric$ is the Ricci curvature of the metric $g$ at time $t$, $c_1$ and $c_2$ are functions depending only on time $t$ which will be used to get rid of the undesire terms.

Note that the time dependence of various quantities in (\ref{timeRiem}) is suppressed and this convention is used throughout this paper. Note also that the choice of $g$ in (\ref{timeRiem}) is natural from the point of view of the Ricci flow. In fact, if $\bar g$ is a solution of the Ricci flow $\dot{\bar g}=-2\ric$, then
\[
g(t,x)=a_1(t)\bar g(a_2(t),x)
\]
satisfies (\ref{timeRiem}).

We consider the following functional which is a special case of (\ref{fnal})
\begin{equation}\label{ent3}
\E_3(t,\mu)=\int_M\log(\rho(t,x))d\mu(x),
\end{equation}
where $\mu=\rho(t,\cdot)\m_t$, $\m_t=e^{-k(t)u(t,\cdot)}\vol$, $u$ is the solution of the Hamilton-Jacobi equation
\[
\dot u+\ham(t,x,du)=0
\]
with initial condition $u\Big|_{t=0}=\f$ (here $\f$ is the potential which define the displacement interpolation $\mu_t$), and $k$ is another functional parameter.

Next, we choose suitable functions $c_1$, $c_2$, $b$, $U$, and measure $\m_t$ such that the right hand side of (\ref{conclude}) can be estimated. This leads to the following result which is a consequence of Theorem \ref{main} (see the proof for some explanations on the choices of $c_1$, $c_2$, $b$, and $U$).

\begin{cor}\label{main2}
If the functions $c_1, c_2, b, U$ satisfy
\[
c_1k=-2,\quad \ddot k=-b\dot k,\quad 2\dot k=c_2k-bk,\quad U=-\frac{c_1^2}{8}R,
\]
then
\[
\frac{d^2}{dt^2}\E_3(t,\mu_t)+b(t)\frac{d}{dt}\E_3(t,\mu_t)+\frac{n\dot c_2(t)}{2}+\frac{nc_2(t)^2}{4}+\frac{nb(t)^2}{4}\geq 0
\]
for each smooth displacement interpolation $\mu_t$ corresponding to the cost (\ref{prob:cost}) with Hamiltonian $\ham$ given by (\ref{mech}) and (\ref{timeRiem}). Here $R$ denotes the scalar curvature of the given Riemannian metric $g$ at time $t$.
\end{cor}

If we further specialize to the case $k(t)=Ct^m$ and $m\neq 0$, then the rest of the functional parameters ($c_1, c_2, b, U$) are completely determined and Corollary \ref{main2} simplifies to the following result.

\begin{cor}\label{main3}
If $m\neq 0$, $k(t)=Ct^m$, $c_1(t)=-\frac{2}{Ct^{m}}$, $c_2(t)=\frac{m+1}{t}$, $b(t)=-\frac{m-1}{t}$, and $U(t,x)=-\frac{c_1(t)^2}{8}R(t,x)$, then
\[
\frac{d^2}{dt^2}\E_3(t,\mu_t)-\frac{m-1}{t}\frac{d}{dt}\E_3(t,\mu_t)\geq \frac{m(1-m)n}{2t^2}\\
\]
for each smooth displacement interpolation $\mu_t$ corresponding to the cost (\ref{prob:cost}) with Hamiltonian $\ham$ given by (\ref{mech}) and (\ref{timeRiem}).
\end{cor}

If $k(t)\equiv C$ (i.e. $m=0$), then Corollary \ref{main2} gives the following result.

\begin{cor}\label{main3p}
If $k(t)\equiv C\neq 0$, $c_1(t)=-\frac{2}{C}$, $c_2(t)=b(t)$, and $U(t,x)=-\frac{1}{2C^2}R(t,x)$, then
\[
\frac{d^2}{dt^2}\E_3(t,\mu_t)+b(t)\frac{d}{dt}\E_3(t,\mu_t)+\frac{n\dot b(t)}{2}+\frac{nb(t)^2}{2}\geq 0
\]
for each smooth displacement interpolation $\mu_t$ corresponding to the cost (\ref{prob:cost}) with Hamiltonian $\ham$ given by (\ref{mech}) and (\ref{timeRiem}).
\end{cor}

Corollary \ref{main3} and \ref{main3p} contain some known results. If we set $m=-\frac{1}{2}$ and $C=-1$ in Corollary \ref{main3}, then $c_1(t)=2\sqrt t$ and $c_2(t)=\frac{1}{2t}$. A calculation shows that $g=\sqrt t\,\bar g$, where $\bar g$ is a solution of the backward Ricci flow $\dot{\bar g}=2\overline\ric$. Here $\overline\ric$ is the Ricci curvature of the Riemannian metric $\bar g$.

The Lagrangian $\lag$ is given by
\begin{equation}\label{L-1}
\begin{split}
\lag(t,x,v)&=\frac{1}{2}\sum_{i,j}g_{ij}(t,x)v_iv_j+\frac{t}{2}R(t,x)\\
&=\frac{\sqrt t}{2}\left(\sum_{i,j}\bar g_{ij}(t,x)v_iv_j+\bar R(t,x)\right),
\end{split}
\end{equation}
where $\bar R$ is the scalar of the metric $\bar g$. The corresponding action functional
\[
\gamma\mapsto \int_0^T\lag(t,\gamma(t),\dot\gamma(t))dt
\]
is Perelman's $\mathcal L$-functional introduced in \cite{Pe}.

Let $\overline{\vol}$ be the Riemannian volume of the metric $\bar g$. Let $\E_4$ be the following functional
\begin{equation}\label{ent4}
\E_4(t,\mu)=\int_M(\log(\rho(t,x))+t^{-1/2} u(t,x))\,d\mu(x)+\frac{n}{2}\log t,
\end{equation}
where $\mu=\rho(t,\cdot)\,\overline{\vol}$. The following result on $\E_4$ which appeared in \cite{To,Lo} is a special case of Corollary \ref{main3}. It was also shown in \cite{To,Lo} how this result leads to the monotonicity of Perelman's reduced volume \cite{Pe}.

\begin{cor}\label{main4}\cite[Theorem 1]{Lo}
The functional $\E_4$ defined in (\ref{ent4}) satisfies
\[
\frac{d^2}{dt^2}\E_4(t,\mu_t)+\frac{3}{2t}\frac{d}{dt}\E_4(t,\mu_t)\geq 0
\]
for each smooth displacement interpolation $\mu_t$ corresponding to the cost (\ref{prob:cost}) with Lagrangian $\lag$ given by (\ref{L-1}).
\end{cor}

Similarly, if we set $m=-\frac{1}{2}$ and $C=1$ in Corollary \ref{main3}, then $c_1(t)=-2\sqrt t$ and $c_2(t)=\frac{1}{2t}$. In this case, the time dependent metric $g$ satisfies $g=\sqrt t\,\bar g$, where $\bar g$ is a solution of the Ricci flow $\dot{\bar g}=-2\overline\ric$. The Lagrangian $\lag$ is given by
\begin{equation}\label{L1}
\lag(t,x,v)=\frac{\sqrt t}{2}\left(\sum_{i,j}\bar g_{ij}(t,x)v_iv_j+\bar R(t,x)\right).
\end{equation}

Let $\overline{\vol}$ be the Riemannian volume of the metric $\bar g$. Let $\E_5$ be the following functional
\begin{equation}\label{ent5}
\E_5(t,\mu)=\int_M(\log(\rho(t,x))-t^{-1/2} u(t,x))\,d\mu(x)+\frac{n}{2}\log t,
\end{equation}
where $\mu=\rho(t,\cdot)\,\overline{\vol}$.

The following result on $\E_5$ is also a consequence of Corollary \ref{main3}. This is another result in \cite{Lo} which leads to the monotonicity of Feldman-Ilmanen-Ni's forward reduced volume \cite{FeIlNi}.

\begin{cor}\label{main5}\cite[Proposition 20]{Lo}
The functional $\E_5$ defined in (\ref{ent5}) satisfies
\[
\frac{d^2}{dt^2}\E_5(t,\mu_t)+\frac{3}{2t}\frac{d}{dt}\E_5(t,\mu_t)\geq 0
\]
for each smooth displacement interpolation $\mu_t$ corresponding to the cost (\ref{prob:cost}) with Lagrangian $\lag$ given by (\ref{L1}).
\end{cor}

If we set $C=1$ and $b\equiv 0$ in Corollary \ref{main3p}, then $g$ is a solution of the Ricci flow $\dot g=-2\ric$. The Lagrangian $\lag$ is given by
\begin{equation}\label{L0}
\lag(t,x,v)=\frac{1}{2}\left(\sum_{i,j}g_{ij}(t,x)v_iv_j+R(t,x)\right).
\end{equation}

Let $\E_6$ be the following functional
\begin{equation}\label{ent6}
\E_6(t,\mu)=\int_M(\log(\rho(t,x))-u(t,x))\,d\mu(x),
\end{equation}
where $\mu=\rho(t,\cdot)\,\vol$. The following result on $\E_6$, which is a consequence of Corollary \ref{main3p}, appeared in \cite{McTo,Lo}.

\begin{cor}\label{main6}\cite[Corollary 4]{Lo}
The functional $\E_6$ defined in (\ref{ent6}) satisfies
\[
\frac{d^2}{dt^2}\E_6(t,\mu_t)\geq 0
\]
for each smooth displacement interpolation $\mu_t$ corresponding to the cost (\ref{prob:cost}) with Lagrangian $\lag$ given by (\ref{L0}).
\end{cor}

We also obtain some similar results for functionals given by (\ref{fnal}) with $\cF(r)=r^q$, where $q$ is a constant. Using the above mentioned approach, we prove the following result which holds for all fibrewise strictly convex Hamiltonians $\ham$ defined above.

\begin{thm}\label{newmain1}
Assume that $\cF(r)=r^q$. Then the functional $\F$ defined in (\ref{fnal}) satisfies the following inequality
\[
\begin{split}
&\frac{d^2}{dt^2}\F(t,\mu_t)+(qb_1(t)+b_2(t))\frac{d}{dt}\F(t,\mu_t)+\frac{q\,b_1(t)^2}{4}+\frac{nb_2(t)^2}{4}\\
&\geq \int_M q\,(r^t_\x)^{q}\Big[\tr(\R^t_{\Phi_t(\x,d\f)})-b_2(t)\,\frac{d}{dt}v(t,\Phi_t(\x,d\f))\\
&\quad\quad -\frac{d^2}{dt^2}v(t,\Phi_t(\x,d\f))\Big] d\mu_0(\x)
\end{split}
\]
for each smooth displacement interpolation $\mu_t$ corresponding to the cost (\ref{prob:cost}).
\end{thm}

In the Riemannian case, the convexity of the functional (\ref{fnal}) with $\cF(r)=r^q$ along displacement interpolations is equivalent to the, so called, curvature-dimension condition. This result, which appeared in \cite{St}, can be obtained as a consequence of Theorem \ref{newmain1}. More precisely, let $\E_7$ be the functional defined by
\begin{equation}\label{ent7}
\E_7(\mu)=-\int_M\rho(x)^{-1/N}\,d\mu(x),
\end{equation}
where $\mu=\rho\,\vol$.

\begin{cor}\cite[Corollary 1.6]{St}
The Ricci curvature $\ric$ of the Riemannian metric $g$ is non-negative and the dimension of the manifold $M$ is $\leq N$ if and only if
\begin{equation}\label{Riccibdd}
\frac{d^2}{dt^2}\E_7(\mu_t)\geq 0
\end{equation}
for each smooth displacement interpolation $\mu_t$ corresponding to the cost (\ref{prob:cost}) with Lagrangian defined by (\ref{kinetic}).
\end{cor}

Next, we specialize Theorem \ref{newmain1} to the case of the Ricci flow. Let $\ham$ be the Hamiltonian defined in (\ref{mech}) with the metric $g$ satisfying (\ref{timeRiem}). Let $\E_8$ be the functional defined by
\[
\E_8(t,\mu)=\int_M\rho(t,x)^qd\mu(x)
\]
where $\mu=\rho(t,\cdot)\m_t$, $\m_t=e^{-k(t)u(t,\cdot)}\vol$, and $k$ is a functional parameter. In this case, Theorem \ref{newmain1} gives to the following result.

\begin{cor}\label{newmain2}
If the functions $c_1, c_2, b, U$ satisfy
\[
c_1k=-2,\quad \ddot k=-b_2\dot k,\quad 2\dot k=c_2k-b_2k,\quad U=-\frac{c_1^2}{8}R,
\]
then
\[
\begin{split}
&\frac{d^2}{dt^2}\E_8(t,\mu_t)+(q\,b_1(t)+b_2(t))\frac{d}{dt}\E_8(t,\mu_t)\\
&+q\left(\frac{n\dot c_2(t)}{2}+\frac{nc_2(t)^2}{4}+\frac{qb_1(t)^2}{4}+\frac{nb_2(t)^2}{4}\right)\E_8(t,\mu_t)\geq 0
\end{split}
\]
for each smooth displacement interpolation $\mu_t$ corresponding to the cost (\ref{prob:cost}) with Hamiltonian $\ham$ given by (\ref{mech}) and (\ref{timeRiem}).
\end{cor}

We remark that, unlike Corollary \ref{main2}, there are two functional parameters, say $k$ and $b_1$, which are free in Corollary \ref{newmain2}. Finally, if we specialize to the case $k(t)=C_1t^m$ and $b_1(t)=C_2t^{-1}$, then we obtain the following corollaries.

\begin{cor}\label{newmain3}
If $m\neq 0$, $k(t)=C_1t^m$, $b_1(t)=C_2t^{-1}$, $c_1(t)=-\frac{2}{C_1t^{m}}$, $c_2(t)=\frac{m+1}{t}$, and $U(t,x)=-\frac{c_1^2}{8}R(t,x)$, then
\[
\begin{split}
&\frac{d^2}{dt^2}\E_8(t,\mu_t)+\frac{qC_2-m+1}{t}\frac{d}{dt}\E_8(t,\mu_t)\\
&\quad +\frac{q(2nm(m-1)+qC_2^2)}{4t^2}\E_8(t,\mu_t)\geq 0
\end{split}
\]
for each smooth displacement interpolation $\mu_t$ corresponding to the cost (\ref{prob:cost}) with Hamiltonian $\ham$ given by (\ref{mech}) and (\ref{timeRiem}).
\end{cor}

\begin{cor}\label{newmain3}
If $k\equiv C_1\neq 0$, $b_1(t)=C_2t^{-1}$, $c_1\equiv-\frac{2}{C_1}$, $c_2=b_2\equiv \frac{C_3}{t}$, and $U(t,x)=-\frac{c_1^2}{8}R(t,x)$, then
\[
\begin{split}
&\frac{d^2}{dt^2}\E_8(t,\mu_t)+\left(\frac{qC_2+C_3}{t}\right)\frac{d}{dt}\E_8(t,\mu_t)\\
&+q\left(\frac{qC_2^2+2nC_3^2-2nC_3}{4t^2}\right)\E_8(t,\mu_t)\geq 0
\end{split}
\]
for each smooth displacement interpolation $\mu_t$ corresponding to the cost (\ref{prob:cost}) with Hamiltonian $\ham$ given by (\ref{mech}) and (\ref{timeRiem}).
\end{cor}

Finally, we also consider the following Hamiltonian $\ham$ motivated by the recent work \cite{Le3} of the author
\[
\ham(t,x,p)=\frac{1}{2}\sum_{i,j}g^{ij}(x)p_ip_j+\sum_{i,j}g^{ij}p_iW_{x_j}(t,x)+U(t,x)
\]
where $g$ is a Riemannian metric, $U$ and $W$ are time-dependent potentials on the manifold $M$. The corresponding Lagrangian $\lag$ is given by
\begin{equation}\label{lagdrift}
\begin{split}
\lag(t,x,v)&=\frac{1}{2}\sum_{i,j}g_{ij}(x)(v_i-\sum_sg^{is}W_{x_s})(v_j-\sum_lg^{jl}W_{x_l})-U(t,x)\\
&=\frac{1}{2}|v-\nabla W|^2-U(t,x).
\end{split}
\end{equation}

\begin{thm}\label{neweg}
Assume that the Riemannian metric $g$ has non-negative Ricci curvature and the potentials $U$ and $W$ satisfy the following condition
\[
\Delta\dot W+\frac{1}{2}\Delta|\nabla W|^2-\Delta U\leq 0.
\]
Then
\[
\begin{split}
&\frac{d^2}{dt^2}\mathcal E_1(\mu_t)\geq 0
\end{split}
\]
for each smooth displacement interpolation $\mu_t$ corresponding to the cost (\ref{prob:cost}) with Lagrangian given by (\ref{lagdrift}).
\end{thm}

\bigskip

\section{Curvature of Hamiltonian Systems}\label{COHS}

In this section, we give a brief discussion on the curvature of a Hamiltonian system introduced in \cite{AgGa}. For a more detail  discussion, see for instance \cite{Ag,Le2}.

Let $T^*M$ be the cotangent bundle of a $n$-dimensional manifold $M$. Let $(x_1,...,x_n)$ and $(x_1,...,x_n,p_1,...,p_n)$ denote systems of local coordinates on $M$ and $T^*M$, respectively. Let $\omega:=\sum_{i=1}^ndp_i\wedge dx_i$ be the canonical symplectic form. Let $\ham:\Real\times T^*M\to\Real$ be a smooth function, called Hamiltonian, and let $\vec\ham$ be the corresponding Hamiltonian vector field defined by $\omega(\vec\ham,\cdot)=-d\ham(\cdot)$, where $d$ denotes the exterior differential on $T^*M$. In local coordinates, $\vec\ham$ is given by
\begin{equation}\label{Ham}
\dot x_i=\ham_{p_i},\quad \dot p_i=-\ham_{x_i}.
\end{equation}
Recall that the dots here and throughout this paper denotes the derivative with respect to time variables $t$, $s$, or $\tau$. We assume that the (possibly time-dependent) vector field $\vec\ham$ is complete and the corresponding (possibly time-dependent) flow is denoted by $\Phi_{t,s}$ (i.e. $\frac{d}{dt} \Phi_{t,s}=\vec\ham(\Phi_{t,s})$ and $\Phi_{s,s}$ is the identity map).

Let $\pi:T^*M\to M$ be the projection to the base $M$ and let $\ver$ be the vertical bundle defined as the kernel of the map $d\pi$. Let $(\x,\p)$ be a point in $T^*M$ and let $J_{(\x,\p)}^{t,s}$ be the family of subspaces in $T_{(\x,\p)} T^*M$, called Jacobi curve, defined by
\[
J^{t,s}_{(\x,\p)}=d\Phi_{t,s}^{-1}(\ver_{\Phi_{t,s}(\x,\p)}).
\]

Assume that the Hamiltonian $\ham$ is fibrewise strictly convex (i.e. $p\mapsto \ham(t,x,p)$ is strictly convex for each time $t$ and each point $x$). Then the family of bilinear forms $\left<\cdot,\cdot\right>^{\tau,s}$ defined by
\[
\left<v,v\right>^{\tau,s}:=\omega\left(\frac{d}{dt} e^{t,s},e^{t,s}\right)\Big|_{t=\tau}
\]
is an inner product defined on the subspace $J^{t,s}_{(\x,\p)}$, where $e^{t,s}$ is contained in $J^{t,s}$ for each time $t$ in a neighborhood of $\tau$ and $e^{\tau,s}=v$.

\begin{prop}\label{canonical}
Assume that the Hamiltonian $\ham$ is fibrewise strictly convex. Then there exists a family of bases $e^{t,s}_1,...,e^{t,s}_n$ contained in $J^{t,s}$ orthonormal with respect to the inner product $\left<\cdot,\cdot\right>^{t,s}$ such that $\ddot e^{t,s}_i$ is contained in $J^{t,s}$. Moreover, if $\bar e^{t,s}_1,...,\bar e^{t,s}_n$ is another such family, then there is an orthonormal matrix $O$ such that
\[
\bar e^{t,s}_i=\sum_{j=1}^nO_{ij}e^{t,s}_j.
\]
\end{prop}

For the proof of the above proposition, see \cite[Proposition 3.2]{Le2}. We define $f^{t,s}_i=\dot e^{t,s}_i$ and call the frame
\[
(e^{t,s}_1,...,e^{t,s}_n,f^{t,s}_1,...,f^{t,s}_n)^T
\]
a canonical frame of $J^{t,s}$. Recall that there are two time parameters $t$ and $s$ for the canonical frame $e^{t,s}_i$ and $f^{t,s}_i$. The dots in $\dot e^{t,s}_i$ and $\dot f^{t,s}_i$ always denote the derivative with respect to the first time parameter, which is $t$ in this case. Let $\bar E^{t,s}=(\bar e^{t,s}_1,...,\bar e^{t,s}_n)^T$ be a family of orthonormal bases of $J^{t,s}$. The following gives a procedure for obtaining a canonical frame. See \cite[Lemma 3.5]{Le2} for a proof.

\begin{prop}\label{compute}
Let $\Omega^{t,s}$ be the matrix defined by $\Omega^{t,s}_{ij}=\omega(\dot{\bar e}^{t,s}_i, \dot{\bar e}^{t,s}_j)$ and let $O^{t,s}$ be a solution of the equation
\[
\dot O^{t,s}=\frac{1}{2}O^{t,s}\Omega^{t,s}.
\]
Then $O^{t,s}\bar E^{t,s}$ defines a canonical frame of $J^{t,s}$.
\end{prop}

If $(e^{t,s}_1,...,e^{t,s}_n,f^{t,s}_1,...,f^{t,s}_n)^T$ is a canonical frame of the Jacobi curve $J^{t,s}_{(\x,\p)}$, then the operators $\R^{t,s}_{(\x,\p)}:J^{t,s}_{(\x,\p)}\to J^{t,s}_{(\x,\p)}$ defined by
\[
\R^{t,s}_{(\x,\p)} e_i^{t,s}=-\ddot e_i^{t,s}
\]
are the curvature operators of the Jacobi curve $J^{t,s}$. The operator $\R^t_{(\x,\p)}:=\R^{t,t}_{(\x,\p)}:\ver_{(\x,\p)}\to\ver_{(\x,\p)}$ is the curvature operator of the Hamiltonian system $\vec\ham$.

The operators $\R^{t,s}$ are symmetric with respect to the inner product $\left<\cdot,\cdot\right>^{t,s}$. It also follows from the definition of $\R$ that the following holds
\[
\R^{t,\tau}_{\Phi_{\tau,s}(\x,\p)}=d\Phi_{\tau,s}\R_{(\x,\p)}^{t,s}(d\Phi_{\tau,s})^{-1}.
\]
In particular, the following holds
\begin{equation}\label{curvconn}
\R^t_{\Phi_{t,s}(\x,\p)}=d\Phi_{t,s}\R_{(\x,\p)}^{t,s}(d\Phi_{t,s})^{-1}
\end{equation}
and it shows that the curvature $\R^{t,s}$ of the Jacobi curve $J^{t,s}$ is completely determined by the curvature $\R^t$ of the Hamiltonian system $\vec\ham$.

Next, we discuss how to compute the curvature operators. Let $V_1^{t},...,V_n^{t}$ be a family of (local) vector fields contained in the vertical bundle $\ver$ which is orthonormal with respect to the Riemannian metric on $\ver$ defined by
\[
v\mapsto \sum_{i,j}\ham_{p_ip_j}(t,x,p)v_iv_j.
\]
A computation using local coordinates shows that this Riemannian metric is also given by $v\mapsto\omega([\vec\ham,V],V)$, where $V$ is any extension of $v$ to a vector field contained in $\ver$.

Let $\bar e_i^{t,s}$ be defined by
\[
\bar e_i^{t,s}=(\Phi_{t,s})^*V_i^t.
\]
It follows that
\[
\begin{split}
\left<\bar e_i^{t,s},\bar e_j^{t,s}\right>^{t,s}_{(\x,\p)}&=\omega_{(\x,\p)}(\dot{\bar e}_i^{t,s},\bar e_j^{t,s})\\
&=\omega_{\Phi_{t,s}(\x,\p)}([\vec\ham,V_i^t],V_j^t)\\
&=\delta_{ij}.
\end{split}
\]

Therefore, we can apply Proposition \ref{compute} and obtain the following result.

\begin{prop}\label{Rformula}
Let $[\R^t_{(\x,\p)}]$ be the matrix with $ij^{th}$ entry defined by $\left<\R^t_{(\x,\p)} V_i^t,V_j^t\right>$. Then it satisfies
\[
\begin{split}
&[\R^t_{(\x,\p)}]\bar E^{t,t}=-\frac{1}{4}(\Omega^{t,t})^2\bar E^{t,t}-\frac{1}{2}\dot\Omega^{t,t}\bar E^{t,t}-\Omega^{t,t}\dot{\bar E}^{t,t}-\ddot{\bar E}^{t,t},\\
&\dot{\bar e}_i^{t,t}=[\vec\ham,V_i^t]+\dot V_i^t,\\
&\ddot{\bar e}_i^{t,t}=[\vec\ham,[\vec\ham,V_i^t]]+2[\vec\ham,\dot V_i^t]+[\vec{\dot\ham},V_i^t]+\ddot V_i^t,\\
&\Omega^{t,t}=\omega(\dot{\bar e}_i^{t,t}, \dot{\bar e}_j^{t,t}),\\
&\dot\Omega^{t,t}=\omega(\ddot{\bar e}_i^{t,t}, \dot{\bar e}_j^{t,t})+\omega(\dot{\bar e}_i^{t,t}, \ddot{\bar e}_j^{t,t}),
\end{split}
\]
where $\bar E^{t,t}=(\bar e^{t,t}_1,...,\bar e^{t,t}_n)^T$.
\end{prop}

\begin{proof}
Let $O^{t,s}$ be as in Proposition \ref{compute} with initial condition $O^{s,s}=I$. Then we have
\[
\frac{d}{dt}\left(O^{t,s}\bar E^{t,s}\right)\Big|_{t=s}=\dot{\bar E}^{s,s}+\frac{1}{2}\Omega^{s,s}\bar E^{s,s}
\]
and
\[
\begin{split}
[\R^s_{(\x,\p)}]\bar E^{s,s}&=-\frac{d^2}{dt^2}\left(O^{t,s}\bar E^{t,s}\right)\Big|_{t=s}\\
&=-\frac{d}{dt}\left(\frac{1}{2} O^{t,s}\Omega^{t,s}\bar E^{t,s}+O^{t,s}\dot{\bar E}^{t,s}\right)\Big|_{t=s}\\
&=-\frac{1}{4}(\Omega^{s,s})^2\bar E^{s,s}-\frac{1}{2}\dot\Omega^{s,s}\bar E^{s,s}-\Omega^{s,s}\dot{\bar E}^{s,s}-\ddot{\bar E}^{s,s}.
\end{split}
\]
\end{proof}

\smallskip

\section{Optimal Transportation and Hamilton-Jacobi Equation}\label{OPAHJE}

In this section, we first recall several known facts about solutions of the Hamilton-Jacobi equation and use them to define the smooth displacement interpolations. Then we define a version of Hessian corresponding to a Hamiltonian. We show that the Hessian of a smooth solution to the Hamilton-Jacobi equation satisfies a matrix Riccati equation. To see the importance of this equation, we show that a Bochner type formula for Hamiltonian system holds under extra assumptions.

First, we recall the following result which is an immediate consequence of the method of characteristics. A proof can be found, for instance, in \cite{CaSi,Ev,AgSa}. (Recall that $\Phi_{t,s}$ is the flow of the Hamiltonian vector field $\vec\ham$ and $\Phi_t=\Phi_{t,0}$).

\begin{prop}\label{HJ}
Let $u$ be a smooth solution of the equation
\begin{equation}\label{HJE}
\dot u+\ham(t,x,du_x)=0
\end{equation}
defined on $[0,T]\times M$ with initial condition $u(0,\cdot)=\f(\cdot)$. Then $u$ satisfies the following relation
\begin{equation}\label{HJR}
\Phi_t(x,d\f)=(\varphi_t(x),du(t,\cdot)),
\end{equation}
where $\varphi_t(x)=\pi(\Phi_t(x,d\f))$ and $t$ is in $[0,T]$. Moreover, the path $t\mapsto\varphi_t(x)$ is the unique minimizer of (\ref{prob:cost}) between the endpoints $x$ and $\varphi_T(x)$.
\end{prop}

In fact, one can use the relation (\ref{HJR}) to give the following short time existence of solution to (\ref{HJE}) for any given smooth initial data $\f$. A detail discussion can be found, for instance, in \cite{AgSa}.

\begin{prop}\label{HJexist}
Assume that the manifold $M$ is closed. For any smooth initial data $\f$, there is a time $T>0$ such that the equation (\ref{HJE}) admits a smooth solution on $[0,T]\times M$ with initial condition $u(0,\cdot)=\f(\cdot)$.
\end{prop}

We define smooth displacement interpolations of Problem \ref{prob:OT} as any one parameter family of measures defined by (\ref{disint})
where $\f$ is any smooth function such that (\ref{HJE}) admits a smooth solution on $[0,T]\times M$ with initial condition $u(0,\cdot)=\f(\cdot)$. The terminology is justified by the following proposition. The proof is a simple application of Proposition \ref{HJ} and the Kantorovich duality (see \cite{Vi1,Vi2}). The proof can be found in \cite{Le1}.

\begin{prop}
Let $\mu_t$ be a smooth displacement interpolation corresponding to the cost (\ref{prob:cost}) defined by the function $\f$. Then the map
\[
\varphi_T(x):=\pi(\Phi_T(d\f_x))
\]
is a solution of Problem \ref{prob:OT} between any measure $\mu$ and $(\varphi_T)_*\mu$.
\end{prop}

Next, we introduce Hessian of the function $\f$ with respect to the Hamiltonian system $\vec\ham$. Let us consider the map $x\mapsto d\f_x$. The differential of this map sends the tangent space $T_xM$ to a Lagrangian subspace of the symplectic vector space $T_{(x,d\f)}T^*M$. This is the graph of a linear map from the horizontal space
\[
\hor^t_{(x,d\f)}=\textbf{span}\{f_1^{t,t},...,f_n^{t,t}\}
\]
to the vertical space $\ver_{(x,d\f)}$. The negative of this linear map is the Hessian of $\f$ with respect to the Hamiltonian system $\vec\ham$ at $(t,x)$ and it is denoted by  $\hess_\ham\f(t,x)$. The trace with respect to the canonical frame $E^{t,t}=(e^{t,t}_1,...,e^{t,t}_n)^T$ and $F^{t,t}=(f^{t,t}_1,...,f^{t,t}_n)^T$ at $(x,d\f)$ is denoted by $\Delta_\ham\f(t,x)$.

Let $E^{t,s}=(e^{t,s}_1,...,e^{t,s}_n)^T$ and $F^{t,s}=(f^{t,s}_1,...,f^{t,s}_n)^T$ be a canonical frame of the Jacobi curve $J^{t,s}_{(\x,d\f)}=d\Phi_{t,s}^{-1}(\ver_{\Phi_{t,s}(\x,d\f)})$. Note that the following relation holds
\[
d\Phi_\tau(J^{t,0}_{(\x,d\f)})=J^{t,\tau}_{\Phi_\tau(\x,d\f)}=J^{t,\tau}_{(\varphi_\tau(\x),du(t,\cdot))}.
\]
It follows that $\tilde E^{t,\tau}=(\tilde e^{t,\tau}_1,...,\tilde e^{t,\tau}_n)^T$ and $\tilde F^{t,\tau}=(\tilde f^{t,\tau}_1,...,\tilde f^{t,\tau}_n)^T$ defined by
\[
\begin{split}
&\tilde E^{t,\tau}:=d\Phi_\tau(E^{t,0})=(d\Phi_\tau(e_1^{t,0}),...,d\Phi_\tau(e_1^{t,0}))^T\\
&\tilde F^{t,\tau}:=d\Phi_\tau(F^{t,0})=(d\Phi_\tau(f_1^{t,0}),...,d\Phi_\tau(f_1^{t,0}))^T
\end{split}
\]
is a canonical frame of the Jacobi curve $J^{t,\tau}_{\Phi_\tau(\x,d\f)}$.

Let $A(t)$ and $B(t)$ be the matrices defined by
\begin{equation}\label{AB}
d(d\f)_{\x}(d\pi(f^{0,0}_i))=\sum_{j=1}^n\left(A_{ij}(t)e^{t,0}_j+B_{ij}(t)f^{t,0}_j\right).
\end{equation}
where $d(d\f)_\x$ is the differential of the map $x\mapsto (d\f)_x$ at $x=\x$.

By applying $d\Phi_t$ on both sides of (\ref{AB}), we obtain the following relation
\begin{equation}\label{AB1}
d\Phi_t(d(d\f)_{\x}(d\pi(f^{0,0}_i)))=\sum_{j=1}^n\left(A_{ij}(t)\tilde e^{t,t}_j+B_{ij}(t)\tilde f^{t,t}_j\right).
\end{equation}

It follows from this equation and (\ref{HJR}) that
\[
d(du(t,\cdot))_{\varphi_t(\x)}(d\varphi_t(d\pi(f^{0,0}_i)))=\sum_{j=1}^n\left(A_{ij}(t)\tilde e^{t,t}_j+B_{ij}(t)\tilde f^{t,t}_j\right).
\]

If we apply (\ref{AB1}) to this equation, then we obtain
\[
\sum_{j=1}^nB_{ij}(t)d(du(t,\cdot))_{\varphi_t(\x)}(d\pi(\tilde f_j^{t,t})) =\sum_{j=1}^n\left(A_{ij}(t)\tilde e^{t,t}_j+B_{ij}(t)\tilde f^{t,t}_j\right).
\]

Therefore, the matrix representation $S(t)$ of the Hessian $\hess_\ham u(t,\varphi_t(\x))$ (as a map from the horizontal space $\hor^t_{(\varphi_t(\x),du(t,\cdot))}$ to the vertical space $\ver_{(\varphi_t(\x),du(t,\cdot))}$) with respect to the basis $\tilde E^{t,t}$ and $\tilde F^{t,t}$ is given by
\begin{equation}\label{S}
S(t)=-B(t)^{-1}A(t).
\end{equation}

By differentiating (\ref{AB}) and using the relation (\ref{curvconn}), we also obtain
\[
\dot B(t)=-A(t),\quad \dot A(t)=B(t)[\R^t_{\Phi_t(\x,d\f)}],
\]
where $[\R^t_{\Phi_t(\x,d\f)}]$ denotes the matrix representation of $\R^t_{\Phi_t(\x,d\f)}$ with respect to the basis $E^{t,t}$.

It follows from these equations and (\ref{S}) that the following matrix Riccati equation holds
\begin{equation}\label{Riccati}
\dot S(t)+S(t)^2+[\R^t_{\Phi_t(\x,d\f)}]=0.
\end{equation}

The equation (\ref{Riccati}) can be seen as a Hamiltonian analogue of Bochner formula in Riemannian geometry. For this, we use an argument in \cite{Vi2}. Indeed, assume that the map $f\mapsto \Delta_\ham f$ is linear. Then we can take the trace of the Riccati equation (\ref{Riccati}) and obtain the following
\[
\begin{split}
&-\Delta_\ham\ham(t,\varphi_t(x),du)+d(\Delta_\ham u)(\ham_p(\varphi_t(x)))\\
&=\Delta_\ham\dot u(t,\varphi_t(x))+d(\Delta_\ham u)(\dot\varphi_t(x))\\
&=\frac{d}{dt} (\Delta_\ham u(t,\varphi_t(x)))\\
&=-|\hess_\ham u(t,\varphi_t(x))|^2-\tr(\R^t_{\Phi_t(\x,d\f)}).
\end{split}
\]
Here $|\hess_\ham u(t,\varphi_t(x))|$ denotes the Hilbert-Schmidt norm of the linear operator $\hess_\ham u_t:\hor^t_{(\varphi_t(\x),du(t,\cdot))}\to\ver_{(\varphi_t(\x),du(t,\cdot))}$ with respect to the Euclidean structure defined by a canonical frame $E^{t,t}$ on the vertical space $\ver_{(\varphi_t(\x),du(t,\cdot))}$ and $F^{t,t}$ on the horizontal space $\hor^t_{(\varphi_t(\x),du(t,\cdot))}$.

By setting $t=0$, we obtain the following Bochner type formula
\[
\Delta_\ham\ham(t,x,d\f)-d(\Delta_\ham \f)(\ham_p(x,d\f))=|\hess_\ham \f(x)|^2+\tr(\R^0_{(\x,d\f)}).
\]

\begin{rem}
One can use the matrix Riccati equation (\ref{Riccati}) together with comparison results in (\cite{Ro}) to prove analogues of Laplacian, Hessian, and volume comparison theorems for Hamiltonian systems.
\end{rem}

\begin{rem}
By combining the work in \cite{AgLe} with the above argument, one can recover the subriemannian Bochner formula in \cite{BaGa} (at least in the three dimensional case) in a geometric way.
\end{rem}

\smallskip

\section{Natural Mechanical Hamiltonians}\label{NMH}

In this section, we discuss how to compute the Hessian and the curvature of the natural mechanical Hamiltonians. Let $g$ be a Riemannian metric, let $g_{ij}=g(\partial_{x_i},\partial_{x_j})$, and let $g^{ij}$ be the inverse matrix of $g_{ij}$. Let $U$ be a smooth function on the manifold $M$. The Hamiltonian given in local coordinates by
\begin{equation}\label{mechanical}
\ham(x,p)=\frac{1}{2}\sum_{i,j}g^{ij}(x)p_ip_j+U(t,x)
\end{equation}
is called a natural mechanical Hamiltonian. When $U\equiv 0$, the Hamiltonian flow of $\ham$ coincides with the geodesic flow of $g$. The proof of the following illustrate how Proposition \ref{Rformula} can be applied and it is needed in the later sections. For an alternative proof, see \cite{Ag}.

\begin{prop}\label{mechanicalcurv}
The Hessian $\hess_\ham \f$ and the curvature $\R$ of the Hamiltonian (\ref{mechanical}) are given by
\[
\hess_\ham \f(t,x)=\hess \f(x)
\]
and
\[
\R^t_{(\x,\p)}(V)=Rm_\x(\p,V)\p+\hess\,U_\x(V),
\]
respectively. Here $Rm$ and $\hess$ denote the Riemannian curvature tensor and the Hessian, respectively, of the given Riemannian metric $g$.
\end{prop}

\begin{rem}
In Proposition \ref{mechanicalcurv}, we have made several identifications $T_\x M\cong T^*_\x M\cong \ver_{(\x,\p)}$. The first identification is done by the Riemannian metric $v\mapsto g(v,\cdot)$ and the second one is by $l\mapsto\frac{d}{dt}\left(\p+tl\right)\Big|_{t=0}$. We will make these identifications for the rest of this paper without mentioning.
\end{rem}

\begin{proof}
Let us fix a geodesic normal coordinates $(x_1,...,x_n)$ around a point $\x$ and let $(x_1,...,x_n,p_1,...,p_n)$ be the corresponding coordinates on the cotangent bundle around the point $(\x,d\f)$. The Hamiltonian vector field $\vec\ham$ is given by
\[
\vec\ham=\sum_{i,j}g^{ij}p_j\partial_{x_i}-\sum_k\Big(\frac{1}{2}\sum_{i,j}g^{ij}_{x_k}p_ip_j+U_{x_k}\Big)\partial_{p_k}.
\]

Let $V_m^t=V_m=\sum_{i}\sqrt{g}_{mi}\partial_{p_i}$ which is independent of time $t$. First, we compute the bracket of $\vec\ham$ and $V_m$
\[
[\vec\ham,V_m]=\sum_{i,j,l}g^{ij}p_j(\sqrt{g}_{ml})_{x_i}\partial_{p_l} -\sum_i\sqrt g^{mi}\partial_{x_i}+\sum_{i,l,k}\sqrt{g}_{ml}\Big(g^{il}_{x_k}p_i\Big)\partial_{p_k}.
\]

Using the properties of geodesic normal coordinates, we have $g^{ij}=\delta^{ij}$ and $(\sqrt g_{ij})_{x_k}=0$ at $\x$. Therefore, the followings hold at $(\x,d\f)$
\[
[\vec\ham,V_m] = -\partial_{x_m},\quad \Omega_{ij}^{0,0}=\omega([\vec\ham,V_i],[\vec\ham,V_j])=0.
\]

It follows that the horizontal space at $(\x,d\f)$ is spanned by $\partial_{x_1},...,\partial_{x_n}$. Since
\[
d(d\f)_\x(-\partial_{x_i})=-\sum_j\f_{x_ix_j}\partial_{p_j}-\partial_{x_i},
\]
$\hess_\ham \f$ is given by $\f_{x_ix_j}=\hess\f$ at $(\x,d\f)$.

Since $(\sqrt g_{lk})_{x_ix_r}=\frac{1}{2}(g_{lk})_{x_ix_r}$ at $\x$, a computation shows that the following holds at $\x$
\[
\begin{split}
&[\vec\ham,[\vec\ham,V_m]]\\
& =\frac{1}{2}\sum_{i,j,k}p_ip_j\Big((g_{mk})_{x_ix_j} +(g_{ij})_{x_kx_m}-2(g_{im})_{x_kx_j}\Big)\partial_{p_k} -\sum_k U_{x_kx_m}\partial_{p_k}.
\end{split}
\]
and
\[
\begin{split}
&\omega([\vec\ham,[\vec\ham,V_m]],[\vec\ham,V_l])\\
&=-\frac{1}{2}\sum_{i,j}p_ip_j\Big((g_{ml})_{x_ix_j} +(g_{ij})_{x_lx_m}-2(g_{im})_{x_lx_j}\Big)+U_{x_lx_m}.
\end{split}
\]
It follows from Proposition \ref{Rformula} that
\[
\dot\Omega_{ml}^{0,0}=\sum_{i,j}p_ip_j\Big((g_{im})_{x_lx_j}-(g_{il})_{x_mx_j}\Big)
\]
at $\x$ (note that since $\vec{\dot H}=\vec{\dot U}$ is contained in $\ver$, $[\vec{\dot H},V_i]=0$).

Therefore,
\[
\begin{split}
&[\vec\ham,[\vec\ham,V_m]]+\Omega_{ml}^{0,0}[\vec\ham,V_l]+\frac{1}{2}\dot\Omega_{ml}^{0,0}V_l+\frac{1}{4}\Omega_{mn}^{0,0}\Omega_{nl}^{0,0}V_l\\
&\quad =\frac{1}{2}\sum_{i,j,k}p_ip_j\Big((g_{mk})_{x_ix_j} +(g_{ij})_{x_kx_m} \\
&\quad \quad -(g_{im})_{x_kx_j}-(g_{ik})_{x_mx_j}\Big)\partial_{p_k} -\sum_k U_{x_kx_m}\partial_{p_k}.
\end{split}
\]

Recall that the Riemann curvature tensor $Rm$ satisfies the following equation at $\x$
\[
R_{ijk}^l=\frac{1}{2}\left((g_{ik})_{x_jx_l}-(g_{kj})_{x_ix_l}-(g_{il})_{x_jx_k}+(g_{lj})_{x_ix_k}\right),
\]
where $R_{ijk}^l$ is defined by $Rm(\partial_{x_i},\partial_{x_j})\partial_{x_k}=\sum_{l}R_{ijk}^l\partial_{x_l}$.

Therefore, by combining all of the above computations with Proposition \ref{Rformula}, we obtain the following equation at $\x$
\[
\R^t_{(\x,\p)}(V)=Rm_\x(\p,V)\p+\hess\,U_\x(V).
\]
\end{proof}

\smallskip

\section{Natural Mechanical Hamiltonians with Evolving Riemannian Metrics}\label{NMHWERM}

In this section, we consider the case when both the Riemannian metric $g$ and the potential $U$ depend on time. Let
\begin{equation}\label{timeHam}
\ham(t,x,p)=\frac{1}{2}\sum_{i,j}g^{ij}(t,x)p_ip_j+U(t,x)
\end{equation}
be the Hamiltonian with Hamiltonian vector field
\[
\vec\ham=\sum_{i,j}g^{ij}p_j\partial_{x_i}-\sum_k\left(\frac{1}{2}\sum_{i,j}g^{ij}_{x_k}p_ip_j+U_{x_k}\right)\partial_{p_k}.
\]

\begin{thm}\label{curvtime}
The Hessian $\hess_\ham \f$ and the curvature $\R$ of the Hamiltonian system $\vec\ham$ defined by (\ref{timeHam}) satisfy
\[
\hess_\ham \f(t,\x)=\hess\f(t,\x)+\frac{1}{2}\dot g(t,\x)
\]
and
\[
\begin{split}
\tr\,(\R_{(\x,\p)}^t) &=\ric(t,\x)(\p,\p)+\Delta U(t,\x)-\left<\nabla(\tr\,(\dot g(t,\x))),\p\right>\\
&+\diver(\dot g(t,\x))(\p)-\frac{1}{2}\tr\,(\ddot g(t,\x))+\frac{1}{4}|\dot g(t,\x)|^2
\end{split}
\]
where $\ric$, $\nabla$, and $\Delta$ are the Ricci curvature, the gradient, and the Laplacian, respectively,  of the Riemannian metric $g(t,\cdot)$ at time $t$.
\end{thm}

\smallskip

\begin{proof}[Proof of Theorem \ref{curvtime}]
We apply Proposition \ref{Rformula} to the time dependent Hamiltonian (\ref{timeHam}). Let us fix a geodesic normal coordinate around $\x$ with respect to the Riemannian metric $g$ at time $t$. A computation shows that
\[
[\vec \ham,\sum_i a_i^t(x)\partial_{p_i}]=\sum_{i,j,l}g^{lj}p_j(a_i^t)_{x_l}\partial_{p_i}-\sum_{j,l}g^{lj}a_j^t\partial_{x_l}+\sum_{j,k,l}g^{lj}_{x_k}p_ja_l^t\partial_{p_k}.
\]
Therefore, the following equation holds at $(\x,\p)$
\begin{equation}\label{brack}
[\vec \ham,\sum_i a_i^t(x)\partial_{p_i}]=\sum_{i,l}p_l(a_i^t)_{x_l}\partial_{p_i}-\sum_l a_l^t\partial_{x_l}.
\end{equation}

Let $V_i^t=\sum_j\sqrt{g}_{ij}\partial_{p_j}$ and let $\bar e^{t,s}_i=(\Phi_{t,s})^*V_i^t$. Then clearly $\bar e^{t,t}_i=\partial_{p_i}$ at $(\x,\p)$. Since $\dot{\sqrt g}_{ij}=\frac{1}{2}\dot g_{ij}$ at $\x$, it follows from (\ref{brack}) that
\[
\bar f_i^{t,t}:=\dot{\bar e}^{t,t}_i=[\vec\ham,V_i^t]+\dot V_i^t=-\partial_{x_i}+\sum_j\frac{1}{2}\dot g_{ij}\partial_{p_j}
\]
at $(\x,\p)$.

It follows that
\[
\begin{split}
d(d\f)_\x(-\partial_{x_i})&=-\sum_j f_{x_ix_j}\partial_{p_j}-\partial_{x_i}\\
&=-\sum_j\left(f_{x_ix_j}+\frac{1}{2}\dot g_{ij}\right)\partial_{p_j}-\partial_{x_i}+\frac{1}{2}\sum_j\dot g_{ij}\partial_{p_j}
\end{split}
\]
and so
\[
\hess_\ham\f=\hess\f+\frac{1}{2}\dot g.
\]

Since $\dot g_{ij}$ is symmetric, it follows that
\[
\Omega_{ij}^{t,t}=\omega(\bar f_i^{t,t},\bar f_j^{t,t})=0.
\]

Clearly, we have
\[
\dot V_i^t=\sum_k\dot{\sqrt g}_{ik}\partial_{p_k}=\frac{1}{2}\sum_k\dot g_{ik}\partial_{p_k}
\]
at $\x$. It follows from (\ref{brack}) that
\[
[\vec\ham,\dot V_i^t]=\frac{1}{2}\sum_{l,m}p_l(\dot g_{im})_{x_l}\partial_{p_m}-\sum_m\frac{1}{2}\dot g_{im}\partial_{x_m}
\]
at $\x$.

By a computation together with the facts that $\dot g_{ij}=-\dot g^{ij}$ and $(\dot g_{ij})_{x_k}=-(\dot g^{ij})_{x_k}$, we see that
\[
[\vec{\dot \ham}_t,V_i^t]=\sum_l\dot g_{il}\partial_{x_l}-\sum_{l,k}(\dot g_{il})_{x_k}p_l\partial_{p_k}
\]
at $\x$.

Since $\ddot{\sqrt g}_{ik}=\frac{1}{2}\ddot g_{ik}-\sum_j\frac{1}{4}\dot g_{ij}\dot g_{jk}$, we have
\[
\ddot V_i^t=\sum_k\ddot{\sqrt g}_{ik}\partial_{p_k}=\sum_k\left(\frac{1}{2}\ddot g_{ik}-\frac{1}{4}\sum_j\dot g_{ij}\dot g_{jk}\right)\partial_{p_k}
\]
at $\x$.

By combining all of the above computations, we obtain
\[
\begin{split}
\dot{\bar f}^{t,t}_i&=[\vec\ham,[\vec\ham,V_i^t]]+2[\vec\ham,\dot V_i^t]+\ddot V_i^t+[\vec{\dot\ham}_t,V_i^t]\\ &=[\vec\ham,[\vec\ham,V_i^t]]+\sum_k\Big(\sum_lp_l(\dot g_{ik})_{x_l}-\sum_l(\dot g_{li})_{x_k}p_l\\
&+\frac{1}{2}\ddot g_{ik}-\frac{1}{4}\sum_j\dot g_{ij}\dot g_{jk}\Big)\partial_{p_k}.
\end{split}
\]

From the proof of Proposition \ref{mechanicalcurv}, the vector field $[\vec\ham,[\vec\ham,V_i^t]]$ is contained in the vertical space $\ver_{(\x,\p)}$ at $(\x,\p)$. Therefore, it follows from the above computations that
\[
\begin{split}
\omega(\ddot{\bar e}_i^{t,t},\dot{\bar e}_j^{t,t}) &=\omega([\vec\ham,[\vec\ham,V_i^t]],[\vec\ham,V_j^t])-\sum_l p_l(\dot g_{ij})_{x_l}\\
 &\quad -\frac{1}{2}\ddot g_{ij}+\frac{1}{4}\sum_k\dot g_{ik}\dot g_{kj}+\sum_l(\dot g_{li})_{x_j}p_l.
\end{split}
\]

Hence,
\[
\begin{split}
\dot\Omega_{ij}^{t,t}&=\omega(\ddot{\bar e}_i^{t,t},\dot{\bar e}_j^{t,t})-\omega(\ddot{\bar e}_j^{t,t},\dot{\bar e}_i^{t,t})\\
&=\omega([\vec\ham,[\vec\ham,V_i^t]],[\vec\ham,V_j^t])-\omega([\vec\ham,[\vec\ham,V_j^t]],[\vec\ham,V_i^t])\\
&\quad+\sum_l(\dot g_{li})_{x_j}p_l-\sum_l(\dot g_{lj})_{x_i}p_l.
\end{split}
\]

Finally, from the proof of Proposition \ref{mechanicalcurv} again, we obtain the following equation at $(\x,\p)$
\[
\begin{split}
-\R^t_{(\x,\p)}&=\frac{1}{2}\dot\Omega^{t,t}\tilde E^{t,t}+\dot{\tilde F}^{t,t}\\
&=\sum_k\Big(-\sum_{i,r}R_{klr}^ip_ip_r-U_{x_lx_k}+\sum_l(\dot g_{ik})_{x_l}p_l-\frac{1}{2}\sum_l(\dot g_{li})_{x_k}p_l\\
&\quad -\frac{1}{2}\sum_l(\dot g_{lk})_{x_i}p_l+\frac{1}{2}\ddot g_{ik}-\frac{1}{4}\sum_j\dot g_{ij}\dot g_{jk}\Big)\partial_{p_k}
\end{split}
\]
and
\[
\begin{split}
\textbf{tr}(\R^t_{(\x,\p)})&=\sum_{l,r,i}R_{llr}^ip_ip_r+\sum_l U_{x_lx_l}\\
&-\left(\sum_{i,l}(\dot g_{ii})_{x_l}p_l-\sum_{i,l}(\dot g_{li})_{x_i}p_l+\frac{1}{2}\sum_i\ddot g_{ii}-\frac{1}{4}\sum_{i,j}\dot g_{ij}\dot g_{ji}\right).
\end{split}
\]
\end{proof}

Finally, we look at the following special case.

\begin{thm}\label{Ricciflow}
Let us fix two smooth functions $c_1,c_2:\Real\to\Real$ of time $t$. Let $g$ be a smooth solution of the following equation
\begin{equation}\label{Ricci}
\dot g=c_1\ric+c_2g.
\end{equation}
Then
\[
\begin{split}
&\tr\,(\R_{(\x,\p)}^t)=\ric(t,\x)(\p,\p)+\Delta U(t,\x)-\frac{1}{2}c_1(t)\left<\nabla R(t,\x),\p\right>\\
&-\frac{\dot c_1(t)}{2}R(t,\x)+\frac{c_1(t)^2}{4}|\ric(t,\x)|^2+\frac{c_1(t)^2}{4}\Delta R(t,\x)-\frac{n\dot c_2(t)}{2}-\frac{nc_2(t)^2}{4},
\end{split}
\]
where $R$ is the scalar curvature of the metric $g$ at time $t$.
\end{thm}

\begin{proof}
By taking the trace of (\ref{Ricci}), we obtain
\[
\tr(\dot g)=c_1R+c_2n.
\]

On the other hand, if we take the covariant derivative of (\ref{Ricci}), then
\[
\nabla_{v_i}\dot g(\p,v_i)=c_1\nabla_{v_i}\ric(\p,v_i).
\]

Therefore, if we sum over $i$ and use the twice contracted Bianchi identity \cite[(1.19)]{ChLuNi}, then
\[
\diver(\dot g)=\frac{1}{2}c_1\,dR.
\]

If we take the divergence again, then we obtain
\[
\diver(\diver(\dot g))=\frac{1}{2}c_1\Delta R.
\]

By using (\ref{Ricci}) again, we obtain
\[
\left<\dot g,\ric\right>=\left<c_1\,\ric+c_2\,g,\ric\right>=c_1|\ric|^2+c_2R.
\]

Therefore, by combining the above calculations with \cite[Lemma 2.7]{ChLuNi}, we have
\begin{equation}\label{dotR}
\begin{split}
\dot R&=-\Delta(\tr(\dot g))+\diver(\diver(\dot g))-\left<\dot g,\ric\right>\\
&=-\frac{c_1}{2}\Delta R-c_1|\ric|^2-c_2R.
\end{split}
\end{equation}

A simple calculation also gives
\[
|\dot g|^2=|c_1\ric+c_2g|^2=c_1^2|\ric|^2+2c_1c_2\,R+n c_2^2.
\]

Therefore, the above calculations reduce Theorem \ref{curvtime} to
\[
\begin{split}
\tr\,(\R_{(\x,\p)}^t) &=\ric_\x(\p,\p)+\Delta U(\x)-\frac{1}{2}c_1\left<\nabla R,\p\right>\\
&-\frac{1}{2}\left(\dot c_1R+c_1\dot R+n\dot c_2\right)-\frac{1}{4}(c_1^2|\ric|^2+2c_1c_2R+n c_2^2)\\
&=\ric_\x(\p,\p)+\Delta U(\x)-\frac{1}{2}c_1\left<\nabla R,\p\right>\\
&-\frac{\dot c_1}{2}R-\frac{n\dot c_2}{2}-\frac{nc_2^2}{4}+\frac{c_1^2}{4}|\ric|^2+\frac{c_1^2}{4}\Delta R.
\end{split}
\]
\end{proof}

\smallskip

\section{Proof of Theorem \ref{main} and Corollary \ref{cor1}}\label{POT}

In this section, we give the proof of Theorem \ref{main} and show that Corollary \ref{cor1} can be seen as a consequence. Before this, we need to introduce the volume distortion factor $v$ in the statement of the theorem. It is defined as follows
\[
v(t,x,p)=\log\left((\pi^*\m_t)_{(x,p)}(f^{t,t}_1,...,f^{t,t}_n)\right).
\]
Clearly, the definition is independent of the choice of a canonical frame due to Proposition \ref{canonical}.

Let $\cF:\Real\to\Real$ and recall that the functional $\F$ is defined by
\[
\F(t,\mu)=\int_M\cF(\rho(t,x))d\mu(x),
\]
where $\mu=\rho(t,\cdot)\m_t$.

Let $\f:M\to\Real$ be a smooth function which defines the smooth displacement interpolation $\mu_t$. Recall that this means
\[
\mu_t=(\varphi_t)_*\mu
\]
where $\varphi_t(x)=\pi(\Phi_t(d\f_x))$. Let $u$ be the solution of the equation (\ref{HJE}) with initial condition $u(0,\cdot)=\f(\cdot)$. First we have the following lemma.

\begin{lem}\label{Fpp}
The functional $\F$ satisfies
\[
\begin{split}
&\frac{d}{dt}\F(t,\mu_t)\\
&=-\int_M \cF'(r^t_\x) r^t_\x\left(\frac{d}{dt}v(t,\Phi_t(\x,d\f))+\Delta_\ham u(t,\varphi_t(\x))\right)d\mu(\x)
\end{split}
\]
and
\[
\begin{split}
&\frac{d^2}{dt^2}\F(t,\mu_t)\\
&=\int_M \left(\cF''(r^t_\x)r^t_\x+\cF'(r^t_\x)\right)r^t_\x\Big(\frac{d}{dt}v(t,\Phi_t(\x,d\f))+\Delta_\ham u(t,\varphi_t(\x))\Big)^2\\
&\quad +\cF'(r^t_\x)r^t_\x\Big(|\hess_\ham u(t,\varphi_t(x))|^2+\tr(\R^{t}_{\Phi_{t}(\x,d\f)})-\frac{d^2}{dt^2}v(t,\Phi_t(\x,d\f))\Big) d\mu(\x).
\end{split}
\]
\end{lem}

\begin{proof}
Let $E^{t,s}=(e^{t,s}_1,...,e^{t,s}_n)^T$ be a canonical frame at $(\x,d\f)$ and let $\tilde E^{t,\tau}=(\tilde e^{t,\tau}_1,...,\tilde e^{t,\tau}_n)^T$ be the canonical frame at $\Phi_t(\x,d\f)$ defined by
\[
\tilde e^{t,\tau}_i=d\Phi_\tau(e^{t,0}_i).
\]
Let $A(t)$ and $B(t)$ be as in (\ref{AB}). It follows from (\ref{AB1}) that
\[
d\varphi_t(d\pi(f^{0,0}_i))=\sum_{j=1}^nB_{ij}(t)d\pi(\tilde f^{t,t}_j).
\]

Let $\rho(t,\cdot)$ be the function defined by $\mu_t:=(\varphi_t)_*\mu=\rho(t,\cdot)\m_t$. The two functions $\rho(0,\cdot)$ and $\rho(t,\cdot)$ are related by
\[
\rho(0,\x)\,e^{v(0,\x,d\f)}=\rho(t,\varphi_t(\x))\det B(t)\,e^{v(t,\Phi_t(\x, d\f))}.
\]

If we differentiate this with respect to $t$, then we obtain
\[
\dot r^t_\x=-r^t_\x\left(\frac{d}{dt}v(t,\Phi_t(\x,d\f))+\tr\left(S(t)\right)\right),
\]
where $r^t_\x=\rho(t,\varphi_t(\x))$ and $S(t)=-B(t)^{-1}A(t)$.

It follows that
\[
\begin{split}
\frac{d}{dt}\F(t,\mu_t)&=\frac{d}{dt}\int_M \cF(\rho(t,\varphi_t(\x)))\ d\mu(\x)\\
&=\int_M \cF'(r^t_\x)\,\dot r^t_\x\ d\mu(\x)\\
&=-\int_M \cF'(r^t_\x) r^t_\x\left(\frac{d}{dt}v(t,\Phi_t(\x,d\f))+\tr\left(S(t)\right)\right)d\mu(\x)
\end{split}
\]
and the first statement of the lemma follows.

If we differentiate the above with respect to time $t$, then we obtain
\[
\begin{split}
&\frac{d^2}{dt^2}\F(t,\mu_t)\\
&=\int_M\cF''(r^t_\x)(\dot r^t_\x)^2 +\cF'(r^t_\x)\ddot r^t_\x d\mu(\x)\\
&=\int_M \left(\cF''(r^t_\x)r^t_\x+\cF'(r^t_\x)\right)r^t_\x\left(\frac{d}{dt}v(t,\Phi_t(\x,d\f))+\tr(S(t))\right)^2\\
& \quad +\cF'(r^t_\x)r^t_\x\left(|S(t)|^2+\tr(\R^{t}_{\Phi_t(\x,d\f)})-\frac{d^2}{dt^2}v(t,\Phi_t(\x,d\f)\right) d\mu(\x)
\end{split}
\]
as claimed.
\end{proof}

\begin{proof}[Proof of Theorem \ref{main}]
By Lemma \ref{Fpp}, we have
\[
\begin{split}
&\frac{d^2}{dt^2}\F(t,\mu_t)\\
&\geq \int_M|S(t)|^2+\tr(\R^{t}_{\Phi_t(\x,d\f)})-\frac{d^2}{dt^2}v(t,\Phi_t(\x,d\f)) d\mu(\x)\\
&\geq \int_Mb(t)\tr(S(t))-\frac{nb(t)^2}{4}+\tr(\R^{t}_{\Phi_t(\x,d\f)})-\frac{d^2}{dt^2}v(t,\Phi_t(\x,d\f))\, d\mu(\x).
\end{split}
\]

Therefore, by Lemma \ref{Fpp} again, we obtain
\[
\begin{split}
&\frac{d^2}{dt^2}\F(t,\mu_t)+b(t)\frac{d}{dt}\F(t,\mu_t)\\
&\geq \int_M\tr(\R^{t}_{\Phi_t(\x,d\f)})-\frac{nb(t)^2}{4}-\frac{d^2}{dt^2}v(t,\Phi_t(\x,d\f))\\
 &\quad-b(t)\frac{d}{dt}v(t,\Phi_t(\x,d\f)) d\mu(\x).
\end{split}
\]
\end{proof}

Next, we prove Corollary \ref{cor1}. In fact, we will first prove results which work for more general Hamiltonian systems for which Corollary \ref{cor1} follows as a consequence. First, we need the following lemma.

\begin{lem}\label{lem}
Assume that the Hamiltonian $\ham$ satisfies $\ham(\x,\lambda\p)=\lambda^m\ham(\x,\p)$. Then
\begin{enumerate}
\item $\pi\Phi_{\lambda^{m-1} t}(x,p)=\pi\Phi_{t}(x,\lambda p)$,
\item $\tr(\R_{(x,\lambda p)})=\lambda^{2m-2}\tr(\R_{(x,p)})$,
\item $\mathcal L^2_{\vec\ham}v(\Phi_t(x,\lambda p))=\lambda^{2m-2}\mathcal L^2_{\vec\ham}v(\Phi_{\lambda^{m-1}t}(x,p))$,
\item $|\hess_\ham (\lambda \f)(x)|^2=\lambda^{2m-2}|\hess_\ham \f(x)|^2$
\end{enumerate}
\end{lem}

\begin{proof}
Let $(x(t),p(t))$ be a solution of the equation (\ref{Ham}). By using homogeneity of $\ham$, one can show that
\[
\left(x(\lambda^{m-1} t),\lambda p(\lambda^{m-1} t)\right)
\]
is also a solution of (\ref{Ham}) and the first claim follows.

Let $(x_1,...,x_n,p_1,...,p_n)$ be local coordinates of a chart on $T^*M$ and let $c_{ij}$ be functions defined on the this chart such that
\[
\partial_{x_i}+\sum_j c_{ij}(x,p)\partial_{p_j}
\]
are contained in the horizontal bundle.

By \cite{Ag}, we have the following formula
\[
\begin{split}
&\sum_{m,s}\ham_{p_ip_m}c_{ms}\ham_{p_sp_j}\\
& =\frac{1}{2}\sum_{k}\left(\ham_{p_k}\ham_{p_ip_jx_k} -\ham_{x_k}\ham_{p_ip_jp_k}-\ham_{p_ix_k}\ham_{p_kp_j}-\ham_{p_ip_k}\ham_{x_kp_j}\right).
\end{split}
\]

It follows from homogeneity of $\ham$ that
\[
c_{ij}(x,\lambda p)=\lambda c_{ij}(x,p).
\]
By Proposition \ref{Rformula} (see also \cite{Ag}),
\[
\R_{(x,p)}(\partial_{p_i})=-[\vec \ham,[\vec \ham,\partial_{p_i}]_\hor]_\ver.
\]
Here the subscripts $\hor$ and $\ver$ denote the horizontal and vertical parts of the corresponding vectors, respectively.

Since $\ham$ is homogeneous in $p$, $\vec\ham$ is horizontal (see \cite{AgGa}). It follows that $\vec\ham=\sum_j\ham_{p_j}(\partial_{x_j}+\sum_kc_{jk}\partial_{p_k})$ and so
\[
\begin{split}
&\R_{(x,p)}(\partial_{p_i})=[\sum_l\ham_{p_l}(\partial_{x_l}+\sum_sc_{ls}\partial_{p_s}),\sum_j\ham_{p_ip_j}(\partial_{x_j}+\sum_kc_{jk}\partial_{p_k})]_\ver\\
&=\sum_{j,l}\ham_{p_ip_j}\ham_{p_l}[\partial_{x_l}+\sum_sc_{ls}\partial_{p_s},\partial_{x_j}+\sum_kc_{jk}\partial_{p_k}]_\ver\\
&=\sum_{j,l}\ham_{p_ip_j}\ham_{p_l}((c_{jk})_{x_l}-(c_{lk})_{x_j}+\sum_s(c_{ls}(c_{jk})_{p_s}-c_{js}(c_{lk})_{x_s}))\partial_{p_k}.
\end{split}
\]
Therefore, $\R_{(x,\lambda p)}(\partial_{p_i})=\lambda^{2m-2}\R_{(x,p)}(\partial_{p_i})$. The second assertion follows from this. Note that the trace of $\R_{(x,\lambda p)}$ is taking with respect to the inner product defined by $\ham_{p_ip_j}(x,\lambda p)$.

By the definition of horizontal frame $f_1,...,f_n$ and homogeneity of $\ham$, we have
\[
\begin{split}
&(\pi^*\m)_{(x,\lambda p)}(f_1(x,\lambda p),...,f_n(x,\lambda p))\\
&=\lambda^{\frac{n(m-2)}{2}}\m_{x}(\pi_*f_1(x,p),...,\pi_*f_n(x,p)).
\end{split}
\]
It follows that
\[
\begin{split}
&\frac{d^2}{dt^2}v(\Phi_t(x,\lambda p))=\frac{d^2}{dt^2}v(x(\lambda^{m-1} t),\lambda p(\lambda^{m-1} t))\\
&=\frac{d^2}{dt^2}\log\left(\m_{x(\lambda^{m-1} t)}(\pi_*f_1(x(\lambda^{m-1} t),p(\lambda^{m-1} t)),...,\pi_*f_n(x(\lambda^{m-1} t),p(\lambda^{m-1} t)))\right)\\
&=\lambda^{2m-2}\mathcal L^2_{\vec\ham}v(x(\lambda^{m-1} t),p(\lambda^{m-1} t)).
\end{split}
\]
This gives the third assertion.

Finally, $\nabla_\ham \f$ is the map which sends $\partial_{x_i}+\sum_jc_{ij}(x,d\f)\partial_{p_j}$ to $\sum_j(\f_{x_ix_j}(x)-c_{ij}(x,d\f))\partial_{p_j}$. Therefore, the final assertion follows.
\end{proof}

\begin{thm}\label{sidethm}
\begin{enumerate}
\item Assume that there is a point $(\x,\p)$ outside the zero section of the cotangent bundle $T^*M$ such that
\[
\tr(\R_{(\x,\p)})-\mathcal L^2_{\vec\ham}v(\x,\p)<\K(\x,\p).
\]
($\mathcal L_V$ denotes the directional derivative with respect to the vector field $V$) Then there is a $\tau >0$ and a smooth potential $\f$ such that the displacement interpolation $\mu_t$ satisfies
\[
\begin{split}
&\frac{d^2}{dt^2}\F(t,\mu_t)<\int_M\K(\Phi_t(\x,d\f)) d\mu(\x)
\end{split}
\]
for all $t$ in $[0,\tau]$.
\item Assume that $\ham$ is homogeneous of degree $m$ in the $\p$ variable and there is a point $(\x,\p)$ outside the zero section of $T^*M$ such that
\[
\tr(\R_{(\x,\p)})-\mathcal L^2_{\vec\ham}v(\x,\p)<K\ham^{\frac{2m-2}{m}}(\x,\p).
\]
Then, for all $T>0$, there is a smooth potential $\f$ such that the displacement interpolation $\mu_t$ satisfies
\[
\begin{split}
&\frac{d^2}{dt^2}\F(t,\mu_t)<K\int_M\ham^{\frac{2m-2}{m}}(x,d\f) d\mu(\x).
\end{split}
\]
\item Assume that $\ham$ is homogeneous of degree $2$ in the $\p$ variable and there is a point $(\x,\p)$ outside the zero section of $T^*M$ such that
\[
\tr(\R_{(\x,\p)})-\mathcal L^2_{\vec\ham}v(\x,\p)<K\ham(\x,\p).
\]
Then, for all $T>0$, there is a smooth potential $\f$ such that the displacement interpolation $\mu_t$ satisfies
\[
\begin{split}
&\frac{d^2}{dt^2}\F(t,\mu_t)<\frac{K}{T}C_T(\mu_0,\mu_T).
\end{split}
\]
\end{enumerate}
\end{thm}

\begin{proof}
Let $(\x,\p)$ be a point in $T^*M$ such that
\[
\tr(\R_{(\x,\p)})-\mathcal L^2_{\vec\ham}v(\x,\p)<\K(\x,\p).
\]
Let $\f$ be a function with compact support such that $\p=df_\x$ and $\hess_\ham \f(\x,\p)=0$. This is possible since it is equivalent to finding a function with prescribed first and second derivatives at a point. Let $u$ be a smooth solution of (\ref{HJE}) defined on $[0,T]\times M$. By Proposition \ref{HJexist}, $u$ exists if we assume that $T>0$ is small enough. By choosing a smaller $T$, we can find a neighborhood $\mathcal O$ of $\x$ such that
\[
|\hess_\ham u(t,\varphi_t(x))|^2+\tr(\R_{\Phi_t(x,d\f)})-\mathcal L^2_{\vec \ham}v(\Phi_t(x,d\f))<\K(\Phi_t(x,d\f))
\]
for all $0\leq t\leq T$ and all $x$ in $\mathcal O$.

Let $\mu$ be a probability measure supported in $\mathcal O$. It follows from Lemma \ref{Fpp} that
\[
\begin{split}
&\frac{d^2}{dt^2}\F(t,\mu_t)<\int_M\K(\Phi_t(\x,d\f)) d\mu(\x).
\end{split}
\]

Next, assume that the Hamiltonian $\ham$ is homogeneous in the $p$ variable. Let $u$ be a smooth solution of the equation (\ref{HJE}) on the time interval $[0,\tau]$ with initial condition $u(0,\cdot)=\f(\cdot)$. Then, by homogeneity of $\ham$, $\bar u(t,x):=(\tau/T)^{\frac{1}{m-1}} u(t\tau/T,x)$ is also a solution on the interval $[0,T]$ with initial condition $\bar u(0,x)=(\tau/T)^{\frac{1}{m-1}} \f(x)$. By Lemma \ref{lem}, it follows that
\[
\bar\varphi_t(x):=\pi(\Phi_t(\lambda^{\frac{1}{m-1}} d\f_x))=\pi(\Phi_{\lambda t}(d\f_x))=\varphi_{\lambda t}(x)
\]
where $\lambda=\tau/T$.

By Lemma \ref{lem}, we also have
\[
\tr\left(\R_{\Phi_t(\x,\lambda^{\frac{1}{m-1}} d\f)}\right)=\tr\left(\R_{(x(\lambda t),\lambda^{\frac{1}{m-1}} p(\lambda t))}\right)=\lambda^2\tr(\R_{\Phi_{\lambda t}(x,d\f)})
\]

By combining the above discussions and applying Lemma \ref{lem} again, it follows that
\[
\begin{split}
&|\hess_\ham \bar u(t,\bar \varphi_t(x))|^2+\tr\left(\R_{\Phi_t(x,\lambda^{\frac{1}{m-1}}d\f)}\right)-\mathcal L^2_{\vec \ham}v(\Phi_t(x,\lambda^{\frac{1}{m-1}}d\f))\\
&=\lambda^2\left(|\hess_\ham u(\lambda t,\varphi_{\lambda t}(x))|^2+\tr(\R_{\Phi_{\lambda t}(x,d\f)})-\mathcal L^2_{\vec \ham}v(\Phi_{\lambda t}(x,d\f))\right)\\
&<\lambda^2K\ham^{\frac{2m-2}{m}}(x,d\f)=K\ham^{\frac{2m-2}{m}}(x,\lambda^{\frac{1}{m-1}}d\f)
\end{split}
\]
for all $x$ in $\mathcal O$ and for all $t$ in $[0,T]$.

Therefore, if we let $\mu$ be a measure supported in $\mathcal O$ and let $\mu_t$ be the displacement interpolation defined by the potential $\lambda^{\frac{1}{m-1}}\f$, then
\[
\begin{split}
&\frac{d^2}{dt^2}\F(t,\mu_t)<K\int_M\ham^{\frac{2m-2}{m}}(x,\lambda^{\frac{1}{m-1}}d\f) d\mu(\x).
\end{split}
\]

If $m=2$, then $\ham(x,\lambda^{\frac{1}{m-1}}d\f)=\frac{1}{T}c_T(x,\bar\varphi_T(x))$ and so
\[
\begin{split}
&\frac{d^2}{dt^2}\F(t,\mu_t)<\frac{K}{T}C_T(\mu_0,\mu_T).
\end{split}
\]
\end{proof}

\begin{proof}[Proof of Corollary \ref{cor1}]
Since $b\equiv 0$ and $\cF(r)=\log r$ in this case, Theorem \ref{main} reduces to
\[
\begin{split}
&\frac{d^2}{dt^2}\F(t,\mu_t)\geq \int_M\Big(\tr(\R^{t}_{\Phi_t(\x,d\f)})-\frac{d^2}{dt^2}v(t,\Phi_t(\x,d\f))\Big) d\mu(\x).
\end{split}
\]

Since $\m_t=e^{-\mathfrak U}\vol$, where $\vol$ is the Riemannian volume form, it follows from the proof of Proposition \ref{mechanicalcurv} that $\tr\R^{t}_{(x,p)}=\ric_x(p,p)$ and $v(t,x,p)=-\mathfrak U(x)$.

Therefore, we obtain
\[
\frac{d^2}{dt^2}v(t,\Phi_t(x,p))=-\frac{d^2}{dt^2}\mathfrak U(\varphi_t(x))=-\hess \mathfrak U(\varphi_t(x)).
\]

Let $\exp$ be the Riemannian exponential map and note that
\[
\exp(T\nabla\f(\x))=\pi(\Phi_T(d\f_\x))=\varphi_T(\x).
\]
It follows that
\[
d(\x,\varphi_T(\x))=T|\nabla \f_\x|.
\]

By the above considerations and (\ref{BE}), we have
\[
\begin{split}
\frac{d^2}{dt^2}\F(t,\mu_t)&\geq K\int_M |\Phi_t(\x,d\f)|^2d\mu(\x)\\
&=\frac{K}{T^2}\int_M d^2(\x,\varphi_T(\x))d\mu(\x)\\
&=\frac{2K}{T}C_T(\mu,\nu)\\
\end{split}
\]
as claimed.

The converse follows from Theorem \ref{sidethm}.
\end{proof}

\smallskip

\section{Proof of Corollary \ref{main2} and \ref{main4}}

This section is devoted to the proof of Corollary \ref{main2} and \ref{main4}.

\begin{proof}[Proof of Corollary \ref{main2}]
From the proof of Theorem \ref{curvtime}, we see that $\pi^*\vol_{(x,p)}(f_1^{t,t},...,f_n^{t,t})=1$. Assume $k=k(t)$ depends only on $t$ and $u$ satisfies the Hamilton-Jacobi equation $\dot u+\frac{1}{2}|\nabla u|^2+U=0$ with initial condition $u_{t_0}=\textbf f$. Then
\begin{equation}\label{dv}
\begin{split}
&\frac{d}{dt}v(t,\Phi_{t}(\x,d\f))\\
&=\frac{d}{dt}\left(k(t)u(t,\varphi_t(\x))\right)\\
&=k(t)\dot u(t,\varphi_t(\x))+k(t)\left<\nabla u(t,\varphi_t(\x)),\dot\varphi_t(\x)\right>+\dot k(t)u(t,\varphi_t(\x))\\
&=k(t)\left(\frac{1}{2}|\nabla u(t,\varphi_t(\x))|^2_{\varphi_t(\x)}-U(t,\varphi_t(\x))\right)+\dot k(t)u(t,\varphi_t(\x)).
\end{split}
\end{equation}
and
\begin{equation}\label{ddv}
\begin{split}
&\frac{d^2}{dt^2}v(t,\Phi_{t}(\x,d\f))=\frac{d^2}{dt^2}\left(k(t)u(t,\varphi_t(\x))\right)\\
&=-k(t)\Big(\dot U(t,\varphi_t(\x))+2\left<\nabla U(t,\varphi_t(\x)),\nabla u(t,\varphi_t(\x))\right>\\
&+\frac{c_1(t)}{2}\ric(\nabla u(t,\varphi_t(\x)),\nabla u(t,\varphi_t(\x)))+\frac{c_2(t)}{2}|\nabla u(t,\varphi_t(\x))|^2\Big)\\
&+ 2\dot k(t)\left(\frac{1}{2}|\nabla u(t,\varphi_t(\x))|^2-U(t,\varphi_t(\x))\right)+\ddot k(t)u(t,\varphi_t(\x)).
\end{split}
\end{equation}

If we compare (\ref{ddv}) and Theorem \ref{curvtime}, then we can get rid of the term $\ric(\nabla u,\nabla u)$ by setting $c_1k=-2$ and get rid of the term $\left<\nabla R,\nabla u\right>$ in Theorem \ref{curvtime} by setting $U=-\frac{1}{2k^2}R$. Therefore, (\ref{ddv}), (\ref{dotR}), and Theorem \ref{curvtime} together gives
\begin{equation}\label{almost}
\begin{split}
&\tr\,(\R_{(\x,\p)}^t)-\frac{d^2}{dt^2}v(t,\Phi_t(\x,d\f))=-\frac{n\dot c_2}{2}-\frac{nc_2^2}{4}-\ddot k(t)u(t,\varphi_t(\x))\\
&+\left(c_2(t)k(t)-2\dot k(t)\right)\left(\frac{1}{2}|\nabla u(t,\varphi_t(\x))|^2+\frac{1}{2k(t)^2}R(t,\varphi_t(\x))\right)
\end{split}
\end{equation}
and
\begin{equation}\label{dv2}
\begin{split}
&\frac{d}{dt}v(t,\Phi_{t}(\x,d\f))=\dot k(t)u(t,\varphi_t(\x))\\
&+k(t)\left(\frac{1}{2}|\nabla u(t,\varphi_t(\x))|^2+\frac{1}{2k(t)^2}R(t,\varphi_t(\x))\right).
\end{split}
\end{equation}

Next, we compare (\ref{almost}) and (\ref{dv2}). We get rid of the term $u(t,\varphi_t(\x))$ in both expressions by choosing $\ddot k+b\dot k=0$ and we obtain
\[
\begin{split}
&\tr\,(\R_{(\x,\p)}^{t})-\frac{d^2}{dt^2}v(t,\Phi_{t}(\x,d\f))-b(t)\frac{d}{dt}v(t,\Phi_{t}(\x,d\f))=-\frac{n\dot c_2}{2}-\frac{nc_2^2}{4}\\
&+\left(c_2(t)k(t)-2\dot k(t)-b(t)k(t)\right)\left(\frac{1}{2}|\nabla u(t,\varphi_t(\x))|^2+\frac{1}{2k(t)^2}R(t,\varphi_t(\x))\right).
\end{split}
\]

Finally, if we choose $c_2(t)k(t)-2\dot k(t)-b(t)k(t)=0$, then
\[
\begin{split}
&\tr\,(\R_{(\x,\p)}^{t})-\frac{d^2}{dt^2}v(t,\Phi_{t}(\x,d\f))-b(t)\frac{d}{dt}v(t,\Phi_{t}(\x,d\f))=-\frac{n\dot c_2}{2}-\frac{nc_2^2}{4}.
\end{split}
\]
\end{proof}

\begin{proof}[Proof of Corollary \ref{main4}]
Let $\bar g$ be a solution of the Ricci flow $\dot{\bar g}=2\ric$ and let $g=\sqrt t\,\bar g$ (Note that the Ricci curvatures of $g$ and $\bar g$ are the same). Let $\vol$ and $\overline\vol$ be the volume forms of $g$ and $\bar g$, respectively.

The functional $\bar\F$ is defined by
\[
\begin{split}
\bar\F(t,\mu)&=\int_M(\log\bar\rho+t^{-1/2} u)\,d\mu+\frac{n}{2}\log t\\
&=\int_M\log(\bar\rho\, e^{t^{-1/2} u})d\mu+\frac{n}{2}\log t,
\end{split}
\]
where $\mu=\bar\rho\,\overline{\vol}$.

Since $\vol=t^{n/4}\overline\vol$, we have
\[
\rho\,e^{-t^{-1/2}u}\vol=\rho\m_t=\mu=\bar\rho\,\overline{\vol}=t^{-n/4}\bar\rho\vol.
\]

It follows that
\[
\begin{split}
\bar\F(t,\mu)&=\int_M\log(t^{n/4}\,\rho)d\mu+\frac{n}{2}\log t\\
&=\F(t,\mu)+\frac{3n}{4}\log t
\end{split}
\]

Let $m=-\frac{1}{2}$ and let $C=-1$. Then Theorem \ref{main3} gives
\[
\begin{split}
&\frac{d^2}{dt^2}\bar\F(t,\mu_t)+\frac{3}{2t}\frac{d}{dt}\bar\F(t,\mu_t)\\
&=\frac{d^2}{dt^2}\F(t,\mu_t)+\frac{3}{2t}\frac{d}{dt}\F(t,\mu_t)+\frac{3n}{8t^2}\geq 0.
\end{split}
\]
\end{proof}

\smallskip

\section{Proof of Theorem \ref{newmain1}}

\begin{proof}[Proof of Theorem \ref{newmain1}]
We use the notations in Lemma \ref{Fpp}. By Lemma \ref{Fpp}, we have
\[
\frac{d}{dt}\F(t,\mu_t)=-\int_M q (r^t_\x)^q\left(\frac{d}{dt}v(t,\Phi_t(\x,d\f))+\tr(S(t))\right)d\mu(\x)
\]
and
\[
\begin{split}
&\frac{d^2}{dt^2}\F(t,\mu_t)=\int_M (r^t_\x)^{q}\Big[q^2\Big(\frac{d}{dt}v(t,\Phi_t(\x,d\f))+\tr(S(t))\Big)^2\\
&\quad +q \Big(\tr(S(t)^2)+\tr(\R^{t}_{\Phi_{t}(\x,d\f)}) -\frac{d^2}{dt^2}v(t,\Phi_{t}(\x,d\f))\Big)\Big] d\mu(\x)\\
&\geq \int_M (r^t_\x)^{q}\Big[q^2b_1(t)\Big(\frac{d}{dt}v(t,\Phi_t(\x,d\f))+\tr(S(t))\Big)-\frac{q^2b_1(t)^2}{4}\\
&\quad +q \Big(b_2(t)\tr(S(t))-\frac{nb_2(t)^2}{4}+\tr(\R^{t}_{\Phi_{t}(\x,d\f)})\\
 &\quad -\frac{d^2}{dt^2}v(t,\Phi_t(\x,d\f))\Big)\Big] d\mu(\x).
\end{split}
\]

Therefore,
\[
\begin{split}
&\frac{d^2}{dt^2}\F(t,\mu_t)+(qb_1(t)+b_2(t))\frac{d}{dt}\F(t,\mu_t)\\
&\geq \int_M q(r^t_\x)^{q}\Big[\tr(\R^{t}_{\Phi_{t}(\x,d\f)})-b_2(t)\frac{d}{dt}v(t,\Phi_{t}(\x,d\f))\\
&\quad -\frac{d^2}{dt^2}v(t,\Phi_{t}(\x,d\f))-\frac{qb_1(t)^2}{4}-\frac{nb_2(t)^2}{4}\Big] d\mu(\x).
\end{split}
\]
\end{proof}

\begin{proof}[Proof of Corollary \ref{newmain2}]
The way to find the appropriate unknown functions (for instance $c_1$, $c_2$, $U$) is in the proof of Corollary \ref{main2}. Here we simply prove the result directly. But this follows immediately from (\ref{almost}), (\ref{dv2}), and Theorem \ref{newmain1} since
\[
\begin{split}
&\tr(\R^{t}_{\Phi_{t}(\x,d\f)})-b_2(t)\frac{d}{dt}v(t,\Phi_t(\x,d\f)) -\frac{d^2}{dt^2}v(t,\Phi_t(\x,d\f))\\
&=-\frac{n\dot c_2(t)}{2}-\frac{nc_2(t)^2}{4}.
\end{split}
\]
\end{proof}

\smallskip

\section{The Case with Drift}

In this section, we consider the following Hamiltonian
\begin{equation}\label{timeHamdrift}
\ham(t,x,p)=\frac{1}{2}\sum_{i,j}g^{ij}(t,x)p_ip_j+\sum_{i,j}g^{ij}p_iW_{x_j}(t,x)+U(t,x)
\end{equation}
where $g$ is time-dependent Riemannian metric, $U$ and $W$ are time-dependent potentials on the manifold $M$.

\begin{thm}\label{curvtimedrift}
The Hessian $\hess_\ham \f$ and the curvature $\R$ of the Hamiltonian system $\vec\ham$ defined by (\ref{timeHam}) satisfy
\[
\hess_\ham \f(t,\x)=\hess\f(t,\x)+\hess W(t,\x)+\frac{1}{2}\dot g(t,\x)
\]
and
\[
\begin{split}
\tr(\R_{(\x,\p)}^t)&=\ric(\p+\nabla W,\p+\nabla W)-\frac{1}{2}\Delta|\nabla W|^2-\Delta\dot W+\Delta U\\
&-\frac{1}{2}\tr(\ddot g)+\frac{1}{4}|\dot g|^2-\left<\nabla W+\p,\nabla(\tr(\dot g))-\diver(\dot g)\right>
\end{split}
\]
where $\ric$, $\nabla$, and $\Delta$ are the Ricci curvature, the gradient, and the Laplacian, respectively,  of the Riemannian metric $g(t,\cdot)$ at time $t$.
\end{thm}

\smallskip

\begin{proof}[Proof of Theorem \ref{curvtimedrift}]
A computation shows that
\[
\begin{split}
\vec\ham &=\sum_{i,j}g^{ij}(p_j+W_{x_j})\partial_{x_i}\\
&-\sum_k\left(\frac{1}{2}\sum_{i,j}g^{ij}_{x_k}p_ip_j+\sum_{i,j}g^{ij}_{x_k}p_iW_{x_j}+\sum_{i,j}g^{ij}p_iW_{x_jx_k}+U_{x_k}\right)\partial_{p_k}
\end{split}
\]
and
\begin{equation}\label{brackdriftpro}
\begin{split}
[\vec \ham,\sum_i a_i^t(x)\partial_{p_i}]&=\sum_{i,j,l}g^{lj}\left(p_j+W_{x_j}\right)(a_i^t)_{x_l}\partial_{p_i} -\sum_{j,l}g^{lj}a_j^t\partial_{x_l}\\
&+\sum_{j,k,l}\left(g^{lj}_{x_k}p_j+g^{lj}_{x_k}W_{x_j}+g^{lj}W_{x_jx_k}\right)a_l^t\partial_{p_k}.
\end{split}
\end{equation}
Therefore, the following equation holds at $(\x,\p)$
\begin{equation}\label{brackdrift}
\begin{split}
[\vec \ham,\sum_i a_i^t(x)\partial_{p_i}]&=\sum_{i,l}(W_{x_l}+p_l)(a_i^t)_{x_l}\partial_{p_i}\\
&-\sum_l a_l^t\partial_{x_l}+\sum_{k,l}W_{x_lx_k}a_l^t\partial_{p_k}.
\end{split}
\end{equation}

Let $V_i^t=\sum_j\sqrt{g}_{ij}\partial_{p_j}$ and let $\bar e^{t,s}_i=(\Phi_{t,s})^*V_i^t$. Then clearly $\bar e^{t,t}_i=\partial_{p_i}$ at $(\x,\p)$. Since $\dot{\sqrt g}_{ij}=\frac{1}{2}\dot g_{ij}$ at $\x$, it follows from (\ref{brackdrift}) that
\[
\bar f_i^{t,t}:=\dot{\bar e}^{t,t}_i=[\vec\ham,V_i^t]+\dot V_i^t=-\partial_{x_i}+\sum_j\left(\frac{1}{2}\dot g_{ij}+W_{x_ix_j}\right)\partial_{p_j}
\]
at $(\x,\p)$.

It follows that
\[
\begin{split}
&d(d\f)_\x(-\partial_{x_i})=-\sum_j f_{x_ix_j}\partial_{p_j}-\partial_{x_i}\\
&=-\sum_j\left(f_{x_ix_j}+\frac{1}{2}\dot g_{ij}+W_{x_ix_j}\right)\partial_{p_j}-\partial_{x_i}+\sum_j\left(\frac{1}{2}\dot g_{ij}+W_{x_ix_j}\right)\partial_{p_j}
\end{split}
\]
and so
\[
\hess_\ham\f=\hess\f+\frac{1}{2}\dot g+\nabla^2 W.
\]

Since $\dot g+\nabla^2W$ is symmetric, it follows that
\[
\Omega_{ij}^{t,t}=\omega(\bar f_i^{t,t},\bar f_j^{t,t})=0.
\]

It also follows from (\ref{brackdrift}) that
\[
[\vec\ham,\dot V_i^t]=\frac{1}{2}\sum_{l,m}\left((W_{x_l}+p_l)(\dot g_{im})_{x_l}+W_{x_lx_m}\dot g_{il}\right)\partial_{p_m}-\sum_m\frac{1}{2}\dot g_{im}\partial_{x_m}
\]
at $\x$.

By a computation together with the facts that $\dot g_{ij}=-\dot g^{ij}$ and $(\dot g_{ij})_{x_k}=-(\dot g^{ij})_{x_k}$, we see that
\[
[\vec{\dot \ham}_t,V_i^t]=\sum_l\dot g_{il}\partial_{x_l}+\sum_k\left(\dot W_{x_ix_k}-\sum_j\dot g_{ij}W_{x_jx_k}-\sum_{l}(\dot g_{il})_{x_k}(p_l+W_{x_l})\right)\partial_{p_k}
\]
at $\x$.

Since $\ddot{\sqrt g}_{ik}=\frac{1}{2}\ddot g_{ik}-\sum_j\frac{1}{4}\dot g_{ij}\dot g_{jk}$ at $\x$, we have
\[
\ddot V_i^t=\sum_k\ddot{\sqrt g}_{ik}\partial_{p_k}=\sum_k\left(\frac{1}{2}\ddot g_{ik}-\frac{1}{4}\sum_j\dot g_{ij}\dot g_{jk}\right)\partial_{p_k}
\]
at $\x$.

A computation shows that
\[
\begin{split}
[\vec\ham,[\vec\ham,V_i^t]]&=\frac{1}{2}\sum_{j,k,l}\left(p_j+W_{x_j}\right)(p_k+W_{x_k})(g_{il})_{x_jx_k}\partial_{p_l}\\
&-\sum_{s,j,k}(p_s+W_{x_s})(g_{ij})_{x_kx_s}(p_j+W_{x_j})\partial_{p_k}\\
&+\sum_{j,l}W_{x_j}W_{x_ix_jx_l}\partial_{p_l}+\sum_{k,l}W_{x_ix_k}W_{x_kx_l}\partial_{p_l}\\
&+\sum_k\left(\frac{1}{2}\sum_{j,l}(g_{jl})_{x_kx_i}p_jp_l-\sum_{j,s}p_jg^{js}_{x_ix_k}W_{x_s}-U_{x_ix_k}\right)\partial_{p_k}
\end{split}
\]
at $(\x,\p)$.

By combining all of the above computations, we obtain
\[
\begin{split}
&\dot{\bar f}^{t,t}_i=\dot{\bar f}^{t,t}_i+\sum_j\Omega_{ij}^{t,t}\bar f_j^{t,t}\\
&=[\vec\ham,[\vec\ham,V_i^t]]+2[\vec\ham,\dot V_i^t]+\ddot V_i^t+[\vec{\dot\ham}_t,V_i^t]\\
&=\frac{1}{2}\sum_{j,k,l}\left(p_j+W_{x_j}\right)(p_k+W_{x_k})(g_{il})_{x_jx_k}\partial_{p_l}\\
&-\sum_{s,j,k}(p_s+W_{x_s})(g_{ij})_{x_kx_s}(p_j+W_{x_j})\partial_{p_k}\\
&+\sum_{j,l}W_{x_j}W_{x_ix_jx_l}\partial_{p_l}+\sum_{k,l}W_{x_ix_k}W_{x_kx_l}\partial_{p_l}\\
&+\sum_k\left(\frac{1}{2}\sum_{j,l}(g_{jl})_{x_kx_i}p_jp_l-\sum_{j,s}p_jg^{js}_{x_ix_k}W_{x_s}-U_{x_ix_k}\right)\partial_{p_k}\\
&+\sum_{l,m}\left((W_{x_l}+p_l)(\dot g_{im})_{x_l}\right)\partial_{p_m}\\
\end{split}
\]
\[
\begin{split}
&+\sum_k\left(\dot W_{x_ix_k}-\sum_{l}(\dot g_{il})_{x_k}(p_l+W_{x_l})\right)\partial_{p_k}\\
&+\sum_k\left(\frac{1}{2}\ddot g_{ik}-\frac{1}{4}\sum_j\dot g_{ij}\dot g_{jk}\right)\partial_{p_k}.
\end{split}
\]

Therefore, we obtain
\[
\begin{split}
-\tr(\R_{(\x,\p)}^t)&=\sum_{j,k,i}\frac{1}{2}p_jp_k(g_{ii})_{x_jx_k}+\sum_{j,k,i}\frac{1}{2}W_{x_j}W_{x_k}(g_{ii})_{x_jx_k}\\ &+\sum_{j,k,i}p_jW_{x_k}(g_{ii})_{x_jx_k}-\sum_{j,k,i}p_sp_j(g_{ij})_{x_ix_s}-\sum_{j,k,i}p_s(g_{ij})_{x_ix_s}W_{x_j}\\
&-\sum_{j,k,i}W_{x_s}(g_{ij})_{x_ix_s}p_j-\sum_{i,j,s}W_{x_s}W_{x_j}(g_{ij})_{x_ix_s}+\sum_{i,j}W_{x_j}W_{x_ix_jx_i}\\
&+\sum_{i,j}W_{x_ix_j}W_{x_jx_i}+\frac{1}{2}\sum_{i,j,l}(g_{jl})_{x_ix_i}p_jp_l+\sum_{i,j,s}p_j(g_{js})_{x_ix_i}W_{x_s}\\
&-\sum_i U_{x_ix_i}+\sum_{i,l}(W_{x_l}+p_l)((\dot g_{ii})_{x_l}-(\dot g_{il})_{x_i})\\
&+\sum_i\dot W_{x_ix_i}+\sum_i\left(\frac{1}{2}\ddot g_{ii}-\frac{1}{4}\sum_j\dot g_{ij}\dot g_{ji}\right)\\
&=-\ric(\p,\p)-2\ric(\p,\nabla W)\\
&+\left<\nabla\Delta W,\nabla W\right>+|\nabla^2 W|^2+\Delta\dot W-\Delta U\\
&+\frac{1}{2}\tr(\ddot g)-\frac{1}{4}|\dot g|^2+\left<\nabla W+\p,\nabla(\tr(\dot g))-\diver(\dot g)\right>\\
&=-\ric(\p+\nabla W,\p+\nabla W)+\frac{1}{2}\Delta|\nabla W|^2+\Delta\dot W-\Delta U\\
&+\frac{1}{2}\tr(\ddot g)-\frac{1}{4}|\dot g|^2+\left<\nabla W+\p,\nabla(\tr(\dot g))-\diver(\dot g)\right>
\end{split}
\]
as claimed.
\end{proof}

A computation very similar to that of Theorem \ref{Ricciflow} gives the following.

\begin{thm}\label{Ricciflow2}
Let us fix two smooth functions $c_1,c_2:\Real\to\Real$ of time $t$. Let $g$ be a smooth solution of the following equation
\begin{equation}\label{Ricci}
\dot g=c_1\ric+c_2g.
\end{equation}
Then
\[
\begin{split}
&\tr(\R_{(\x,\p)}^t)=\ric(\p+\nabla W,\p+\nabla W)-\frac{1}{2}\Delta|\nabla W|^2-\Delta\dot W+\Delta U\\
&-\frac{\dot c_1}{2}R-\frac{n\dot c_2}{2}-\frac{nc_2^2}{4}+\frac{c_1^2}{4}|\ric|^2+\frac{c_1^2}{4}\Delta R-\frac{c_1}{2}\left<\nabla W+\p,\nabla R\right>
\end{split}
\]
where $R$ is the scalar curvature of the metric $g$ at time $t$.
\end{thm}

Theorem \ref{neweg} follows immediately from Theorem \ref{main} and Theorem \ref{Ricciflow2}.

\smallskip

\section{Finsler manifolds}\label{finsler}

In this section, we discuss the case where the Hamiltonian is induced by a Finsler metric. More precisely, let $F$ be a Finsler metric defined on the tangent bundle $TM$ (i.e. $F(x,v)$ is smooth outside the zero section, positively homogeneous of degree 1, and $F^2$ is strictly convex in $v$). Let $\lag$ be the Lagrangian defined by
\[
\lag(x,v)=\frac{1}{2}F(x,v)^2
\]
and let $\ham$ be the corresponding Hamiltonian as before. We are going to show that the curvature of this Hamiltonian and the Riemann curvature of the Finsler manifold coincide up to an identification of the tangent and the cotangent bundle.

First, let us recall the definition of the Riemann curvature. Here we only give a very brief discussion. The detail can be found, for instance, in \cite{ChSh}. Let $(x_1,...,x_n,v_1,...,v_n)$ be local coordinates around a point $(\x,\ve)$ in the tangent bundle $TM$. The map $\leg:TM\to T^*M$ defined by
\[
\leg(\x,\ve)(\w)=\frac{d}{dt}\lag(\x,\ve+t\w)\Big|_{t=0}
\]
is a diffeomorphism between $TM$ and $T^*M$. It also induces local coordinates $(x_1,...,x_n,p_1,...,p_n)$ on $T^*M$ around the point $(\x,\leg(\x,\ve))$. More precisely, we have $p_i=\lag_{v_i}(x,v)$. By differentiating
\begin{equation}\label{LH}
\sum_jp_jv_j=\ham(x,p)+\lag(x,v),
\end{equation}
we obtain
\[
\sum_k\ham_{p_ip_k}(x,p)\lag_{v_kv_j}(x,v)=\delta_{ij}.
\]
Here and for the rest of this section,  we can consider $p_i$ as a function of $v_1,...,v_n$ or $v_i$ as a function of $p_1,...,p_n$ via the map $\leg$.

Let
\[
G^i=\frac{1}{2}\sum_l\ham_{p_ip_l}\left(\sum_k\lag_{x_kv_l}v^k-\lag_{x_l}\right)\text{ and } N_i^j=\partial_{v_i}G^j.
\]
The functions $N_i^j$ are the connection coefficients of a Ehresmann connection of the bundle $T(TM)$. In other words, $N_i^j$ defines a sub-bundle $HTM$ of $T(TM)$ which is spanned by the vectors
\[
\partial_{x_i}-\sum_j N_i^j\partial_{v_j}, \quad i=1,...,n
\]
and we have $T(TM)=HTM\oplus VTM$, where $VTM=\{V|d\pi(V)=0\}$ and $\pi:TM\to M$ is the natural projection. This splitting is essentially the, so called, Chern connection (see \cite{ChSh} for the detail).

\begin{lem}
The horizontal bundle $\hor$ of the Hamiltonian system $\vec\ham$ and $HTM$ defined above are related by
\[
\leg_*(HTM)=\hor.
\]
\end{lem}

\begin{proof}
We start with the following formula of $c_{ij}$ which can be found in \cite{Ag}
\[
\begin{split}
&\sum_{m,s}\ham_{p_ip_m}c_{ms}\ham_{p_sp_j}\\
& =\frac{1}{2}\sum_{k}\left(\ham_{p_k}\ham_{p_ip_jx_k} -\ham_{x_k}\ham_{p_ip_jp_k}-\ham_{p_ix_k}\ham_{p_kp_j}-\ham_{p_ip_k}\ham_{x_kp_j}\right).
\end{split}
\]

By differentiating (\ref{LH}), we obtain
\begin{enumerate}
\item $\lag_{x_i}(x,v)+\ham_{x_i}(x,p)=0$,
\item $\ham_{p_ix_j}(x,p)+\sum_k\ham_{p_ip_k}(x,p)\lag_{v_kx_j}(x,v)=0$,
\item $\sum_k\ham_{p_ip_k}(x,p)\lag_{v_kv_j}(x,v)=\delta_{ij}$,
\item $\sum_k\ham_{p_ip_k}(x,p)\lag_{v_kv_jv_l}(x,v)\\+\sum_{k,s}\ham_{p_ip_kp_s}(x,p)\lag_{v_kv_j}(x,v)\lag_{v_sv_l}(x,v)=0$,
\item $\sum_k\ham_{p_ip_k}(x,p)\lag_{v_kv_jx_l}(x,v)+\sum_k\ham_{p_ip_kx_l}(x,p)\lag_{v_kv_j}(x,v)\\
+\sum_{k,s}\ham_{p_ip_kp_s}(x,p)\lag_{v_kv_j}(x,v)\lag_{v_sx_l}(x,v)=0$.
\end{enumerate}

It follows that
\[
\begin{split}
c_{ms}& =\frac{1}{2}\sum_{k,i,j}\Big(v_k\lag_{v_mv_i}\ham_{p_ip_jx_k}\lag_{v_sv_j}\\ &-\lag_{v_mv_i}\ham_{x_k}\ham_{p_ip_jp_k}\lag_{v_sv_j}-\lag_{v_mv_i}\ham_{p_ix_s}-\ham_{x_mp_j}\lag_{v_sv_j}\Big).
\end{split}
\]
Therefore,
\[
\begin{split}
&\sum_s c_{ms}\ham_{p_sp_l}\\
& =\frac{1}{2}\sum_{k,i,j}\left(\partial_{v_m}(\lag_{x_s}\ham_{p_sp_l} -v_k\lag_{v_ix_k}(x,v)\ham_{p_ip_l}(x,p))+2\lag_{x_mv_s}\ham_{p_sp_l}\right)\\
&=-\partial_{v_m}G^l+\sum_s \lag_{x_mv_s}\ham_{p_sp_l}=-\partial_{v_m}G^l-\ham_{p_lx_m}.
\end{split}
\]
Hence,
\[
\begin{split}
d\leg^{-1}(\partial_{x_i}+\sum_j c_{ij}\partial_{p_j})&=\partial_{x_i}+\sum_j \ham_{p_jx_i}\partial_{v_j}+\sum_{j,k} c_{ij}\ham_{p_jp_k}\partial_{v_k}\\
&=\partial_{x_i}-\sum_k N_i^k\partial_{v_k}
\end{split}
\]
and the claim follows.
\end{proof}

Next, we recall the definition of the Riemann curvature for a Finsler manifold. Let
\[
\begin{split}
&R^l_{ji}\partial_{v_l}=\left[\partial_{x_i}-N_i^k\partial_{v_k},\partial_{x_j}-N_j^l\partial_{v_l}\right]\\
&=\sum_l(\partial_{x_j}N_i^l-\partial_{x_i}N_j^l+N_i^k\partial_{v_k}N_j^l-N_j^k\partial_{v_k}N_i^l)\partial_{v_l}.
\end{split}
\]

The Riemann curvature $\mathcal R$ is defined by
\[
\mathcal R(\partial_{v_i})=\sum_{j,l}v_jR^l_{ij}\partial_{v_l}.
\]
The following shows that $\mathcal R$ and the curvature operator $\R$ of the Hamiltonian system $\vec\ham$ are essentially the same.

\begin{thm}
\[
\begin{split}
&\R_{(\x,\p)}(\leg_*\partial_{v_i})=\leg_*(\mathcal R(v_i)).
\end{split}
\]
\end{thm}

\begin{proof}
Since $\ham$ is homogeneous, $\vec\ham$ is horizontal (see \cite{AgGa}). It follows that
\[
\vec\ham=\sum_i\ham_{p_i}(\partial_{x_i}+c_{ij}\partial_{p_j}).
\]
Therefore, we have
\[
\begin{split}
&\leg^*\left(\R_{(\x,\p)}(\leg_*\partial_{v_i})\right)=-\leg^*[\vec\ham,[\vec\ham,\leg_*\partial_{v_i}]_\hor]_\ver\\
&=-[\sum_jv_j(\partial_{x_j}-N_j^k\partial_{v_k}),[\sum_jv_j(\partial_{x_j}-N_j^k\partial_{v_k}),\partial_{v_i}]_{HTM}]_{VTM}\\
&=\sum_jv_j[\partial_{x_j}-N_j^k\partial_{v_k},\partial_{x_i}-N_i^k\partial_{v_k}]_{VTM}\\
&=\sum_{l,j}v_jR^l_{ij}\partial_{v_l}=\sum_l R^l_{i}\partial_{v_l}.
\end{split}
\]
Here the subscripts $HTM$ and $VTM$ denote the components of the vector.
\end{proof}

\end{document}